\documentclass[12pt,letter]{amsart}
\usepackage{euler,latexsym, amssymb, amscd, amsfonts, xypic, url}
\usepackage[usenames,dvipsnames]{xcolor}
\usepackage{colonequals}
\usepackage{comment}
\usepackage{amsmath}
\usepackage[margin=1in]{geometry}

\usepackage[colorlinks]{hyperref}
 \hypersetup{
  colorlinks,
  citecolor=Green,
  linkcolor=red,
  urlcolor=magenta}

\newtheorem{thmintro}{Theorem}
\newtheorem{tm}{Theorem}[section]
\newtheorem{pr}[tm]{Proposition}
\newtheorem{lm}[tm]{Lemma}
\newtheorem{co}[tm]{Corollary}

\theoremstyle{definition}
\newtheorem{df}[tm]{Definition}
\newtheorem{rmk}[tm]{Remark}

\newtheorem{ex}[tm]{Example}

\newcommand{\cat}[1]{{\mathbf{#1}}}

\newcommand{\lra}[1]{\langle #1 \rangle}

\newcommand{\Aone}{{\mathbb{A}^{\!1}}}

\newcommand{\KO}{\mathrm{KO}}
\newcommand{\K}{\mathrm{K}}
\newcommand{\GW}{\operatorname{GW}}

\newcommand{\SH}{\mathrm{SH}}

\newcommand{\LaoneAb}{{\mathrm{L}^{\ab}_{\Aone}}}

\newcommand{\nis}{{\textrm{Nis}}}
\newcommand{\Zar}{{\textrm{Zar}}}
\newcommand{\cell}{{\operatorname{cell}}}
\newcommand{\CcellSW}{C^{\operatorname{SW-cell}}}
\newcommand{\PcellSW}{P^{\operatorname{SW-cell}}}

\DeclareMathOperator*{\colim}{colim}

\newcommand{\bA}{\mathbb{A}}
\newcommand{\PP}{\mathbb{P}}
\newcommand{\D}{\mathbb{D}}

\newcommand{\Z}{\mathbb{Z}}
\newcommand{\N}{\mathbb{N}}

\newcommand{\Q}{\mathbb{Q}}

\newcommand{\R}{\mathbb{R}}
\newcommand{\F}{\mathbb{F}}
\newcommand{\C}{\mathbb{C}}
\newcommand{\G}{\mathbb{G}}

\newcommand{\ZZ}{\Z}
\newcommand{\RR}{\R}

\newcommand{\CC}{\C}
\newcommand{\QQ}{\Q}

\newcommand{\Hom}{\operatorname{Hom}} 

\newcommand{\Mor}{\operatorname{Mor}} 
\newcommand{\End}{\operatorname{End}} 

\newcommand{\Dual}{\mathbb{D}}

\newcommand{\Ch}{\cat{Ch}}

\newcommand{\DpSWAone}{D^{\SW}(\cat{Ab}^p_{\Aone}(k))}

\newcommand{\Set}{\cat{Set}}

\newcommand{\Ab}{\cat{Ab}}

\newcommand{\Sheaves}{\cat{Sh}}
\newcommand{\Sm}{\cat{Sm}}

\newcommand\bbb[1]{\ensuremath{{\mathbf{#1}}}}

\newcommand{\A}{\operatorname{A}}
\newcommand{\Tr}{\operatorname{Tr}}
\newcommand{\tr}{\operatorname{tr}}
\newcommand{\ind}{\operatorname{ind}}

\newcommand{\Top}{\mathrm{top}}

\newcommand{\SW}{\operatorname{SW}}
\newcommand{\MW}{\operatorname{MW}}

\newcommand{\Th}{\operatorname{Th}}
\newcommand{\rk}{\operatorname{rank}}
\newcommand{\Spec}{\operatorname{Spec}}

\newcommand{\Res}{\operatorname{Res}}
\newcommand{\disc}{\operatorname{disc}}

\newcommand{\sgn}{\operatorname{sign}}

\newcommand{\dlog}{\operatorname{dlog}}

\newcommand{\cO}{\mathcal{O}}
\newcommand{\et}{{\text{\'et}}}
\newcommand{\ab}{{\rm ab}}

\newcommand{\calC}{{ \mathcal C}}
\newcommand{\calD}{{ \mathcal D}}

\newcommand{\calF}{{ \mathcal F}}

\newcommand{\calH}{{ \mathcal H}}

\newcommand{\calO}{{ \mathcal O}}

\newcommand{\calX}{{\mathcal X}}

\newcommand{\defi}[1]{\textsf{#1}} 

\newcommand{\id}{\mathrm{Id}}

\begin{document}

\pagestyle{plain}
\title{Quadratic enrichment of the logarithmic derivative of the zeta function}

\author{Margaret Bilu}
\email{margaret.bilu@math.u-bordeaux.fr}
\author{Wei Ho}
\email{who.math@gmail.com}
\author{Padmavathi Srinivasan}
\email{padmask@bu.edu}
\author{Isabel Vogt}
\email{ivogt.math@gmail.com}
\author{Kirsten Wickelgren}
\email{kirsten.wickelgren@duke.edu}

\makeatletter
\@namedef{subjclassname@2020}{%
  \textup{2020} Mathematics Subject Classification}
\makeatother


\date{April 29, 2024}

\begin{abstract}
We define an enrichment of the logarithmic derivative of the zeta function of a variety over a finite field to a power series with coefficients in the Grothendieck--Witt group. We show that this enrichment is related to the topology of the real points of a lift. For cellular schemes over a field, we prove a rationality result for this enriched logarithmic derivative of the zeta function as an analogue of part of the Weil conjectures. We also compute several examples, including toric varieties, and show that the enrichment is a motivic measure.
\end{abstract}
\maketitle

\setcounter{tocdepth}{2}
\tableofcontents

\section{Introduction}

Let $X$ be a smooth projective variety over a finite field $\F_q$. The zeta function of $X$ is defined by 
$$\zeta_X(t) = \exp\left(\sum_{m\geq 1} \frac{|X(\F_{q^m})|}{m}t^m\right).$$
By the celebrated Weil conjectures, the zeta function $\zeta_X(t)$ is a rational function that is also related to the topology of the complex points of a lift of $X$ to characteristic zero. Both of these features were explained by Grothendieck's \'{e}tale cohomological interpretation of the zeta function. It is then natural to wonder if $\zeta_X(t)$ also sees the topology of the points of lifts to other characteristic $0$ fields such as $\mathbb{R}$ or $\mathbb{Q}_p$ when appropriate lifts exist.

The aim of this paper is to define and study an enrichment of (the logarithmic derivative of) the zeta function $\zeta_X(t)$, by replacing the point counts with traces in the sense of $\Aone$-homotopy theory. This enrichment is a repackaging of the information contained in the point counts, additionally weighted by appropriate quadratic forms. We show that in many cases, this enriched zeta function satisfies a rationality result coming from a cohomology theory. Moreover, we show that after applying a suitable invariant, this zeta function relates the point counts over finite fields to the topology of the real points of a lift under certain hypotheses. Instead of \'etale cohomology, we use and extend a cohomology theory recently introduced by Morel and Sawant \cite{MorelSawant}, showing a trace formula for this theory.

Let $X$ be a smooth proper variety over a field $k$ (not necessarily finite). Due to work of Hu, Riou, and Ayoub  \cite{HuPicard, riou2005dualite, ayoub2007six}, $X$ is \defi{dualizable} in the $\Aone$-stable homotopy category $\SH(k)$. It follows that an endomorphism $\varphi: X \to X$ has a \defi{trace} $\Tr(\varphi)$ valued in the endomorphisms of the motivic sphere spectrum.  A theorem of Morel \cite{morel2004motivic-pi0} identifies this endomorphism ring with the \defi{Grothendieck--Witt group} $\GW(k)$ of~$k$, defined to be the group completion of isomorphism classes of symmetric nondegenerate bilinear forms on $k$ (see Section \ref{sec:SWintro}, and note that by Hoyois \cite[Footnote 1]{hoyois2015quadratic}, we need not assume $k$ is perfect). We may therefore define the following enrichment to $\GW(k)$ of the logarithmic derivative of the zeta function:

\begin{df}\label{df:dlogA1zeta}
Let $X$ be a smooth proper variety over a field $k$. Let $\varphi: X \to X$  be an endomorphism. The \defi{$\Aone$-logarithmic zeta function of $(X, \varphi)$} is defined by
\[ \dlog \zeta^{\Aone}_{X, \varphi} \colonequals \sum_{m \geq 1} \Tr(\varphi^m) t^{m-1} \in \GW(k)[[t]].\]
\end{df}

\begin{rmk}
Definition~\ref{df:dlogA1zeta} applies more generally to endomorphisms of dualizable objects of $\SH(k)$. One may take (symmetric monoidal) localizations of $\SH(k)$ or $\SH$ of more general base schemes by replacing $\GW(k)$ with the appropriate endomorphisms of the sphere spectrum. For example, the variety $X$ need not be proper. By \cite[Corollary B.2]{RiouAppendix}, a smooth scheme $X$ over a field $k$ is dualizable in the localized $\Aone$-stable homotopy category $\SH(k)_{\Z\left[\frac{1}{p}\right]}$ where $p$ is the characteristic exponent of $k$. For an endomorphism $\varphi: X \to X$ of a smooth scheme over a field $k$, we define
\[
 \dlog \zeta^{\Aone}_{X, \varphi} \colonequals \sum_{m \geq 1} \Tr(\varphi^m) t^{m-1} \in \GW(k)[\tfrac{1}{p}][[t]].
\] For $p$ odd, $\GW(k) \subseteq \GW(k)[\tfrac{1}{p}]$. Further results on duality, e.g., those of Dubouloz--D\'eglise--{\O}stvaer \cite{DDO-stable_homotopy_infty}, widen the class of $X$ for which $ \dlog \zeta^{\Aone}_{X, \varphi}$ is defined.
\end{rmk}

This $\Aone$-logarithmic zeta function recovers the classical zeta function via the rank map, as we now explain.  The connection comes from a realization functor on $\Aone$-homotopy theory: there is a symmetric monoidal stable \'etale realization $r_{\et, \ell}$ from $\SH(k)$ to the derived category of $\ell$-adic \'etale sheaves on the small \'etale site of $k$. It follows that $r_{\et, \ell}(\Tr \varphi) = \Tr(r_{\et, \ell} \varphi)$, which is the integer-valued trace of the usual Weil conjectures. The rank homomorphism, denoted $\rk: \GW(k) \to \Z$, sends the isomorphism class of a bilinear form $\beta: V \times V \to k$ to the dimension of the $k$-vector space $V$.
Via the identification of $\GW(k)$ with the endomorphisms of the sphere spectrum, the \'etale realization map $r_{\et, \ell}: \End(1_{k}) \to \End (r_{\et, \ell}(1_k))$ is identified with the $\rk$ homomorphism. It follows that
\begin{equation}\label{eq:rankAonezeta=zeta}
\rk \dlog \zeta^{\Aone}_{X, \varphi} = \frac{d}{dt} \log \zeta_{X, \varphi},
\end{equation} where $\zeta_{X, \varphi}$ denotes the classical zeta function
\begin{equation}\label{eq:zeta_classical_as_rational_function}
 \zeta_{X, \varphi}(t) = \prod_i (P_{\varphi\vert H^i_{\et}}(t))^{(-1)^{i+1}} \qquad \text{with} \quad P_{\varphi\vert H^i_{\et}}(t): = \det(1-t \varphi \vert H^i_{\et}(X_{k^s};\Z_{\ell})).
\end{equation}

We investigate the additional information recorded in the $\Aone$-logarithmic zeta function. A case of particular interest is when $k$ is finite and $\varphi:X\to X$ is the Frobenius endomorphism. Let $q$ be odd and let $k$ be the finite field $\F_q$ with $q$ elements. Then the Grothendieck--Witt ring $\GW(\F_q)$ is computed as
\begin{equation}\label{eq:GW_presentation}
\GW(\F_q) \cong \frac{\Z[\lra{u}]}{(\lra{u}^2 -1, 2(\lra{u}-1))}
\end{equation} where $u$ is a fixed non-square in $\F_q$ and $\lra{u}$ denotes the class of the bilinear form $k \times k \to k$ sending $(x,y)$ to $uxy$. As a group $\GW(\F_q)$ is isomorphic to $\Z\times \Z/2\Z$, by sending a class in $\GW(k)$ to the pair given by its rank and its discriminant. As above, the rank gives the classical zeta function. The discriminant term can be computed in terms of $\vert X(\F_{q^m}) \vert$ with Hoyois's beautiful enriched Grothendieck--Lefschetz trace formula \cite{hoyois2015quadratic} (see Section~\ref{S:Hoyois-trace-formula}):
\[
\disc \dlog \zeta^{\Aone}_{X, \varphi}  = \sum_{m \geq 1} \left( \sum_{\substack{i|m\\ i\ \text{even}}}\frac{1}{i}\sum_{d|i} \mu(d)|X(\F_{q^{i/d}})|\right) t^{m-1},
\] where $\mu$ denotes the M\"obius function. Combining rank and discriminant yields the expression
\begin{equation}\label{AoneZeta_from_points}
\dlog \zeta^{\Aone}_{X, \varphi}  = \sum_{m \geq 1}\left( \sum_{i \vert m} \alpha(i) \Tr_{\F_{q^i/\F_q}}\lra{1} \right)t^{m-1},
\end{equation} where $\alpha(i)$ denotes the number of points of $X$ with residue field $\F_{q^i}$ and 
$ \Tr_{\F_{q^i}/\F_q}\lra{1} $ is the transfer on $\GW$ (see \eqref{TrGWeq} and \eqref{E:classicaltrace} for the definition and computation of the transfer). Theorem~\ref{T:enrichedzetaviaHoyois} gives a further computation of $\dlog \zeta^{\Aone}_{X, \varphi} $ for the Frobenius $\varphi$ of $X$ over a finite field in terms of point counts. Such formulas are amenable to (computer) computation up to finitely many coefficients of $t^m$, e.g., for elliptic curves; see Section~\ref{sect:elliptic_curves}. 

Both the classical and $\Aone$-logarithmic zeta function are determined by the number of points of \(X\) over \(\F_{q^m}\) for finitely many \(m\). The beauty of the classical zeta function is that it packages these point counts to reveal topological information about the complex points of a lift.  In the same spirit, we show that the $\Aone$-logarithmic zeta function packages these point counts to additionally reveal topological information about the \textit{real} points of a lift, as we now discuss.

\subsubsection*{Connections with topology and real points}

The classical zeta function may be thought of as an algebraic manipulation of the numbers $\vert X(\F_{q^m}) \vert$ related to the topology of $\calX(\C)$. The $\Aone$-logarithmic zeta function~\eqref{AoneZeta_from_points} is similarly an algebraic manipulation of the $\vert X(\F_{q^m}) \vert$, but it is related both to the topology of $\calX(\C)$ and to the topology of $\calX(\R)$.

\begin{ex}
The scheme $\PP^1 \times \PP^1$ over $\R$ has a twist $\Res_{\C/\R} \PP^1$, given by the Weil restriction of scalars. The topology of the complex points of $\PP^1 \times \PP^1$ is unchanged by the twist, but the topology of the real points is changed from a product of two circles to a $2$-sphere
\[
(\PP^1 \times \PP^1)(\R) \simeq S^1 \times S^1 \quad \quad \Res_{\C/\R} \PP^1(\R) \simeq S^2.
\]
If $q$ is a prime congruent to $3$ modulo $4$, then the extension $\F_q \subset \F_{q^2}$ is given by $\F_{q^2} = \F_q[\sqrt{-1}]$ and the $\mathbb{R}$-schemes $\PP^1 \times \PP^1$ and $\Res_{\C/\R} \PP^1$ are lifts to characteristic zero of the varieties $\PP^1 \times \PP^1$ and $\Res_{\F_{q^2}/\F_q} \PP^1$ over $\F_q$, respectively. The change in the topology of the real points of the lifts to characteristic zero is reflected in the $\Aone$-zeta functions: 
\[
\dlog \zeta^{\Aone}_{\PP^1 \times \PP^1, \varphi}  = \frac{d}{dt} \log \frac{1}{(1- t)(1- q_\epsilon^2t)} + \langle -1 \rangle\frac{d}{dt} \log \frac{1}{(1- q_\epsilon t)^2}
\]
\[
\dlog \zeta^{\Aone}_{\Res_{\F_{q^2}/\F_q} \PP^1, \varphi}  = \frac{d}{dt} \log \frac{1}{(1- t)(1- q_\epsilon^2t)} + \langle -u \rangle \frac{d}{dt} \log \frac{1}{1-q_{\epsilon}\langle u \rangle t} + \langle -1 \rangle \frac{d}{dt} \log \frac{1}{1+q_{\epsilon}t},
\]

\noindent Here $u$ is a non-square in $\F_q^*$, and $q_{\epsilon}= \sum_{i=1}^q \langle (-1)^{i-1}\rangle$. The above formulas in fact hold for all $q$. Under the assumption that $q \equiv 3 \pmod{4}$, we may take to be $u=-1$ and the topology of the real and complex points determines these computations. See Examples~\ref{E:productofproj} and \ref{E:twistedproduct} for the computation, and Remark~\ref{R:changeinrealtop} for the connection with the topology of the real points.  
\end{ex}

More generally, we prove results on the relationship between $\dlog \zeta^{\Aone}_{X, \varphi}$ and $\calX(\R)$ in Section~\ref{Section:Rpoints}. There is a signature homomorphism $\sgn: \GW(\R) \rightarrow \Z$ that sends $a\lra{1} + b \lra{-1}$ to $a-b$. 
In the presence of a lift of Frobenius, Theorem~\ref{pr:lift_to_R} and Proposition~\ref{pr:sigdlogA1_R} say that in an appropriate sense, there is a well-defined signature of $\dlog \zeta^{\Aone}_{X,\F_p}$ that can be computed from the action on $\calX(\R)$.
\begin{thmintro}\label{thmintro:dlogZA1-lift-of-Frobenius-smooth-proper-Z} Let $X$ be an $\F_q$ scheme equipped with a lift of Frobenius to an endomorphism $\varphi:\calX \to \calX$ of a smooth and proper $\ZZ[\frac{1}{d}]$-scheme $\calX$ for $d=1$ or $d$ even. Then $\dlog \zeta^{\Aone}_{X,\F_p}$ lifts to $\GW(\ZZ[\frac{1}{d}])[[t]]$ and
\[
\sgn \dlog \zeta^{\Aone}_{X,\F_p} = \sum_i (-1)^{i+1} \frac{d}{dt}\log \det(1-t \varphi(\R)\vert H^i_{\Top}(\calX(\R);\Z)),
\] where $H^i_{\Top}(\calX(\R);\Z)$ denotes the singular cohomology of the real points of $\calX$.
\end{thmintro} For example, Theorem~\ref{thmintro:dlogZA1-lift-of-Frobenius-smooth-proper-Z} can be applied to any smooth, proper, toric variety; see Example~\ref{ex:toric_varieties}. This relationship with the topology of $\calX(\R)$ also produces examples where the $\Aone$-logarithmic zeta functions of varieties equipped with endomorphisms differ even when the classical ranks are equal. See Example~\ref{ex:same_classical_different_enriched_log_zeta}. The proof of Theorem~\ref{thmintro:dlogZA1-lift-of-Frobenius-smooth-proper-Z} for $d=1$ requires deep results on Hermitian K-theory \cite{CDHHLMNNS-3} \cite{BHeta_periodic}. 

In analogy with the classical case, the $\Aone$-logarithmic zeta function allows us to combine information about $\calX(\R)$ with $\vert X(\F_{q^m}) \vert$ for $m=1,2,3,\ldots$.  Combining Theorem~\ref{thmintro:dlogZA1-lift-of-Frobenius-smooth-proper-Z} with \eqref{AoneZeta_from_points} produces non-trivial lower bounds on the Betti numbers of $\calX(\R)$. See Remark~\ref{rmk:dlog_rationality_R_again} and Example \ref{ex:bounds:b_i}. The analogous relationship between $\calX(\C)$ and $\vert X(\F_{q^m}) \vert$ for $m=1,2,3,\ldots$ provided by the Weil conjectures has had stunning applications, e.g., \cite{Ellenberg-Venkatsh-Westerland}. We are limited by the restrictive hypotheses of Theorem~\ref{thmintro:dlogZA1-lift-of-Frobenius-smooth-proper-Z}. In response, we turn to further machinery.

\subsubsection*{Trace formula and rationality} Missing from the discussion so far is an expression for $\Tr(\varphi)$ in terms of a trace on cohomology, analogous to the Grothendieck--Lefschetz trace formula 
\[
\Tr(\varphi^m) = \sum_{i=0}^{2d} (-1)^i \Tr(\varphi^m\vert H^i_{\et}(X_{k^s};\Z_{\ell})).
\] Such formulas relate the classical zeta function of a variety over a finite field to the complex points of a lift  of the variety, without requiring the much more restrictive hypothesis of a lift of Frobenius.

For this, we use Morel and Sawant's $\Aone$-cellular homology \cite{MorelSawant}, currently defined for so-called \defi{cellular schemes}, which are a generalization of the classical notion of varieties with an affine stratification. There are subtleties in directly using Morel and Sawant's $\Aone$-cellular homology, and modifications are necessary for obtaining a theory with a good formalism for traces; we postpone a summary of these technical challenges to the next section where we give an overview of cohomological techniques. In \cite[Remark 2.44]{MorelSawant}, they note the existence of an extension of this theory to a pro-object for all spaces in the sense of $\Aone$-homotopy theory and in particular, for all smooth schemes over a field. Supported by conjectures on Poincar\'e duality for their theory, they conjecture that for smooth projective varieties over a field, this pro-object is in fact constant \cite{MorelTalk}. If their conjectures are true, our work may likewise extend to varieties which are not cellular. In this paper, we proceed for cellular schemes; see Section \ref{sec:cellularschemes} for the definitions.
 
We use the machinery of Morel and Sawant to give a $\GW(k)$-enriched Grothendieck--Lefschetz trace formula for the somewhat more restrictive class of \defi{simple} cellular schemes (see Definition \ref{df:finite_cellular_structure}; these include, e.g., schemes with an affine stratification):

\begin{thmintro}
Let \(\varphi \colon X \to X\) be an endomorphism of a smooth projective simple cellular scheme $X$ over a field $k$. In $\GW(k)$ we have the equality
\[
\Tr(\varphi) =  \sum_i \langle -1 \rangle^i \Tr(C_i^{\cell}(\varphi))
\] where $C_i^{\cell}(\varphi)$ denotes the $\Aone$-cellular complex of Morel and Sawant in degree $i$.
\end{thmintro}
 See Theorem~\ref{tm:Cellular_LTF}. (This is a different result from the main theorem of \cite{hoyois2015quadratic} despite the similarity in the name.) This expression for $\Tr(\varphi)$ in terms of the trace of the cellular complex leads to rationality results for the $\Aone$-logarithmic zeta function, which we now describe.

\

We enrich the logarithmic derivative of the zeta function, rather than the zeta function itself, because the Grothendieck--Witt group in general has torsion elements. (See Remark \ref{rmk:dlogrationality} for why this does not lift to a notion of rationality using a $\lambda$-ring structure.)

\begin{df}\label{df:dlog_rational}
Let $R$ be a ring. We say a power series $\Phi(t) \in R[[t]]$ is \defi{dlog rational} if there exists a finite collection of polynomials $P_j \equiv 1 \pmod t$ in $R[t]$ and elements $c_j \in R$ such that
\[ \Phi(t) = \sum_j c_j \frac{P_j'(t)}{P_j(t)}, \]
where $\frac{1}{P_j(t)} \colonequals \sum_{m \geq 0} (1-P_j(t))^m$.
\end{df}

\begin{ex}
If a rational power series $\Psi$ with integer coefficients can be put in the form
$$\Psi(t) = \prod_{j} P_j(t)^{c_j}$$
where $P_j(t)$ are polynomials with integer coefficients such that $P_j \equiv 1 \mod t$, and $c_j\in \Z$, then $\frac{d}{dt}\log \psi$ is dlog rational. In particular, the Weil conjectures imply that $\frac{d}{dt}\log\zeta_X(t)$ is dlog rational for every smooth projective variety $X$ over a finite field.
\end{ex}

We prove that the $\Aone$-logarithmic zeta functions of simple cellular schemes are dlog rational with polynomial terms coming from the characteristic polynomials of matrices of elements of $\GW(k)$ giving the action of the endomorphism on the cellular complex (see Theorem~\ref{thm:rationality_cellular}):

\begin{thmintro}\label{tm:maintheorem}
Let $k$ be a field, and let $X$ be a smooth projective scheme over $k$ with a simple cellular structure. Let $\varphi:X\to X$ be an endomorphism of $X$. The function $\dlog \zeta_{X,\phi}^{\Aone}$ is dlog rational. More precisely,
\[\dlog \zeta^{\Aone}_{X, \varphi} = \sum_{i = - \infty}^\infty -\langle -1 \rangle^i \frac{d}{dt}\log P_{C^\cell_i(\varphi)}(t), \quad \text{where} \quad  P_{C^\cell_i(\varphi)}(t) = \det (1- t C^\cell_i(\varphi))\] and $C^\cell_i(\varphi)$ is a square matrix of elements of $\GW(k)$.
\end{thmintro}

Our methods also yield variations on such a result. For example, we weaken \defi{simple cellular structure} to \defi{cellular structure} in Theorem~\ref{thm:rationality_cellular_slick}: we lose matrices of elements of $\GW(k)$ and dlog rationality in the sense of Definition~\ref{df:dlog_rational} in exchange for an abstract logarithmic derivative of a characteristic polynomial on the $i$th term of Morel--Sawant's cellular complex. This abstract logarithmic derivative of a characteristic polynomial is introduced in Definition~\ref{df:log_derivative_char_poly_varphi} and is an abstraction of the elementary algebraic lemma that says that if $\varphi:V\to V$ is an endomorphism of a finite-dimensional vector space with characteristic polynomial $P_\varphi(t)$, then
$$\frac{d}{dt} \log P_\varphi(t) = -\sum_{m\geq 1}\Tr(\varphi^m)t^{m-1}.$$ We use the notational convention that $\frac{d}{dt} \log$ denotes taking a derivative of a logarithm, while $\dlog $ denotes a formal substitute.

Note that in this setup, the logarithmic $\Aone$-zeta function still retains terms associated to the cellular complex.

\begin{ex} In the logarithmic zeta functions for both
$\PP^1 \times \PP^1$ and $\Res_{\F_{q^2}/\F_q} \PP^1$ with the Frobenius $\varphi$,
we have terms $\frac{d}{dt}\log ((1-t)^{-1}(1-q_{\epsilon}^2t)^{-1})$, which come from the degree $0$ and $2$ terms in the cellular complex (which up to quasi-isomorphism is independent of the cellular structure! \cite[Corollary 2.42]{MorelSawant}). The difference comes from terms in degree $1$. For $\PP^1 \times \PP^1$ these are straightforward to compute (see Corollary~\ref{co:zeta_strict_cellular} for the computation of the $\Aone$-logarithmic zeta function of any \defi{strictly cellular}, smooth, projective variety, including products of projective spaces and Grassmannians). We obtain a quadratic term $\lra{-1}\frac{d}{dt}\log \frac{1}{(1-q_{\epsilon}t)^2}$ from the two obvious $1$-cells of $\PP^1 \times \PP^1$. The $\lra{-1}$ in front corresponds to the fact that the real dimension of these cells is odd. In $\dlog \zeta^{\Aone}_{\Res_{\F_{q^2}/\F_q} \PP^1, \varphi}$ by contrast, the contribution from the $1$-cells is 
\[
\lra{-u} \frac{d}{dt}\log \frac{1}{(1-q_{\epsilon}\lra{u}t)} + \lra{-1} \frac{d}{dt}\log \frac{1}{(1+q_{\epsilon}t),} 
\] which has vanishing signature (plug in $u=-1$ and $q_\epsilon =1$), caused by the disappearance of the $1$-cells over $\RR$.
\end{ex}

After this paper appeared, Xiaowen Hu showed that for any smooth, proper scheme over a finite field, $\dlog \zeta^{\Aone}_{X}$ is rational, but not always $\dlog$ rational \cite{Hu-rationality_dlogZ}.

\subsubsection*{Overview of cohomological techniques.} We briefly describe the main ideas that are needed to adapt Morel and Sawant's $\Aone$-cellular homology \cite{MorelSawant} to our purposes; developing this theory is the focus of Sections \ref{sec:cellularhomology} and \ref{sec:SWcellcomplex}. The goal is to construct a cohomology theory on appropriate geometric objects for which the traces of endomorphisms lie in $\GW(k)$ and are related to our enriched zeta function via an analogue of a Grothendieck--Lefschetz trace formula. 

We first use Morel--Sawant's $\Aone$-cellular homology to define a symmetric monoidal functor $\tilde{C}_*^\cell$ from an $\Aone$-homotopy category of cellular schemes  $\calH(k)^{\cell}_{\ast}$ to a derived category which we denote $D(\Ab^p_{\Aone}(k))$. The derived category of strictly $\Aone$-invariant sheaves $D(\Ab_{\Aone}(k))$ does not necessarily have enough projectives. We thus pass to a category $D(\Ab^p_{\Aone}(k))$, which is large enough to have a symmetric monoidal structure and receive a symmetric monoidal functor from $\calH(k)^{\cell}_{\ast}$ but small enough that the tensor product is the correct derived tensor product.

We then pass to Spanier--Whitehead categories (by inverting $\PP^1$ and its image) on both sides, which produces a symmetric monoidal functor
$$\CcellSW_\ast: \SW^{\cell}(k) \to D^{\SW}(\Ab^p_{\Aone}(k)).$$
 As smooth projective cellular $k$-schemes become dualizable in $\SW^{\cell}(k)$ (see Lemma \ref{Xsmproj_is_dualizable_SW}), the symmetric monoidal functor $\CcellSW_\ast$ then functorially attaches dualizable objects in $D^{\SW}(\Ab^p_{\Aone}(k))$ to smooth projective cellular $k$-schemes, thus giving the desired trace formalism. As our cohomology theory is given by an explicit chain complex, this allows us to, in many cases, compute traces algebraically, e.g., as traces of explicitly defined matrices.

\subsubsection*{Outline of paper}
Section \ref{sec:preliminaries} introduces some background for the remainder of the paper and constructs a symmetric monoidal derived category of certain strictly $\Aone$-invariant sheaves $D(\Ab^p_{\Aone}(k))$, which we will use for some of our computations of traces. Section \ref{Aone-Spanier-Whitehead-cellular} describes the categories of cellular schemes used in the paper. In Section \ref{sec:cellularhomology}, we use Morel--Sawant's \cite{MorelSawant} $\Aone$-cellular homology to define a symmetric monoidal functor $\tilde{C}_*^\cell$ from an $\Aone$-homotopy category of cellular schemes to $D(\Ab^p_{\Aone}(k))$. (See Proposition~\ref{Ccell_symmetric_monoidal}.) In Section \ref{sec:SWcellcomplex}, we pass to appropriate Spanier--Whitehead categories to obtain dualizable objects (Lemma~\ref{Xsmproj_is_dualizable_SW}, Proposition~\ref{HrespectDtr}). This defines a symmetric monoidal functor $\CcellSW_\ast$ from the $\Aone$-cellular-Spanier--Whitehead category to a $\underline{K}^{\MW}_1[1]$-Spanier--Whitehead category of $D(\Ab^p_{\Aone}(k))$ given in Definition~\ref{df:CcellSW}. We give an explicit computation of the dual of $\CcellSW_\ast(X_+)$ for $X$ a scheme with a simple (e.g., strict) cellular structure; see Propositions~\ref{P:dualinDSWAbAonek} and \ref{pr:TraceDperf}. In Section~\ref{subsection:LTFcellular}, we show the resulting Grothendieck--Lefschetz trace formula. We then prove in Section \ref{sec:rationality} the desired rationality results by computing the trace using $\CcellSW_\ast$. Aspects of the relationship with real points are discussed in Section~\ref{Section:Rpoints}. In Section \ref{sec:computingexamples}, we compute examples. Finally, in Section \ref{sec:motivic_measures}, we explore the link between our $\Aone$-logarithmic zeta function for the Frobenius with Kapranov's motivic zeta function, explaining that they appear to be different. We also show the $\Aone$-logarithmic zeta function defines a motivic measure on the (modified) Grothendieck ring of varieties $K_0(\mathrm{Var}_{\F_q})$. 

\subsubsection*{Acknowledgements}
We wish to thank Piotr Achinger, Aravind Asok, Tom Bachmann, Spencer Bloch, Fr\'ed\'eric D\'eglise, Grigory Garkusha, Bruno Kahn, Hannah Larson, Aram Lindroth, Fabien Morel, Markus Spitzweck, Vasudevan Srinivas, Thor Wittich, and Paul Arne {\O}stv{\ae}r for useful comments and discussions. We are particularly grateful to Tom Bachmann for correcting two related errors in an earlier draft.

We warmly thank Fabien Morel for discussions during the August 2022 Motivic Geometry conference in Oslo. He has independently considered using the Grothendieck--Witt enrichment from $\SH$ in the context of the Weil Conjectures. There are connections of our work to his ongoing work with Anand Sawant. In particular, their conjectures on the $\Aone$-cellular complex may help extend rationality results like Theorem~\ref{thm:rationality_cellular_slick} to arbitrary smooth projective varieties.

This project grew out of the Women in Numbers 5 virtual workshop. WH was partially supported by NSF CAREER DMS-1844763 and the Minerva Research Foundation.
PS was supported by the Simons Collaboration in Arithmetic Geometry, Number Theory, and Computation via the Simons Foundation grant 546235.
IV was partially supported by NSF MSPRF DMS-1902743 and by NSF DMS-2200655.
KW was partially supported by NSF CAREER DMS-2001890 and NSF-DMS 2103838. She also thanks the Isaac Newton Institute for hospitality during the program {\em K-theory, Algebraic Cycles and Motivic Homotopy}.

\section{Preliminaries} \label{sec:preliminaries}

\subsection{Categorical preliminaries}\label{S:categorical}

Let $(\calC, \otimes_{\calC}, 1_{\calC}, \tau_{\calC})$ be a symmetric monoidal category in the sense of, e.g., \cite[1.2]{MartyF09}.

\begin{df}\label{df:dualizable_object_symmetric_monoidal_cat}
(See, for example, \cite[Definition 2.1]{PontoShulman}.) An object $A$ of a symmetric monoidal category $(\calC, \otimes_{\calC}, 1_{\calC}, \tau_{\calC})$ is \defi{dualizable} if there is a dual object $\D A$ in $\calC$ in the sense that there exist coevaluation and evaluation maps
\[ 
\eta: 1_{\calC} \to A \otimes_{\calC} \D A \quad \quad \epsilon:  \D A \otimes_{\calC} A \to 1_{\calC},
\] respectively, such that the composites
$$ A \xrightarrow{\eta \otimes 1_A} A \otimes_{\calC} \D A \otimes_{\calC} A \xrightarrow{1_A \otimes \epsilon} A $$
$$ \D A \xrightarrow{ 1_{\D A} \otimes \eta} \D A \otimes_{\calC} A \otimes_{\calC} \D A \xrightarrow{\epsilon \otimes 1_{\D A}} \D A $$
are the identity maps $1_A$ and $1_{\D A}$, respectively. 
\end{df}

\begin{df}
(See, for example, \cite[Definition 2.2]{PontoShulman}.) For  an endomorphism $\varphi:A \to A$ of a dualizable object $A$ with dual $\D A$, the \defi{categorical trace} $\tr(\varphi)\in \End(1_{\calC})$ is the composition
\begin{equation*}
1_{\calC}\xrightarrow{\eta} A \otimes_{\calC} \D A \xrightarrow{\varphi\otimes_{\calC}\id_{\D A }} A\otimes_{\calC} \D A  \xrightarrow{\tau_{\calC}} \D A \otimes_{\calC} A \xrightarrow{\epsilon} 1_{\calC}. 
\end{equation*}
\end{df}

\begin{df}
A \defi{symmetric monoidal functor} is a functor $F: \calC \to \calD$ between symmetric monoidal categories, together with a natural isomorphism $F(1_{\calC}) \cong 1_{\calD}$ and for all objects $A, B \in \calC$, natural isomorphisms $\iota_{A,B}: F(A) \otimes_{\calD} F(B) \stackrel{\cong}{\to} F(A \otimes_{\calC} B)$ satisfying associativity, unitality, and symmetry;  the last is the condition that the diagrams
\[
\xymatrix{ F(A \otimes_{\calC} B) \ar[r]^{F(\tau_\calC)} & F(B \otimes_{\calC} A) \\
F(A) \otimes_{\calD} F(B) \ar[u]^{\iota_{A,B}}\ar[r]^{\tau_{\calD}} & F(B) \otimes_{\calD} F(A) \ar[u]^{\iota_{B,A}} }
\] commute (see \cite[Chapter XI]{MacLane}).
\end{df}

We record a well-known proposition for expositional clarity. Its proof is straightforward.

\begin{pr}\label{HrespectDtr}
Let $F: \calC \to \calD$ be a symmetric monoidal functor. Let $A$ be a  dualizable object of $\calC$ with dual $\D A$.  Then \begin{enumerate}
\item $F(A)$ is  dualizable with dual $F(\D A)$, and
\item for any endomorphism $\varphi: A \to A$, we have $\tr(F(\varphi)) = F(\tr(\varphi))$.
\end{enumerate}
\end{pr}

We recall the notion of a Spanier--Whitehead category \cite[Section 4]{Voevodsky-ICMA1}.  

\begin{df}\label{df:SWcat}
Let  $(\calC, \otimes_{\calC}, 1_{\calC}, \tau_{\calC})$ be a symmetric monoidal category and $T$ be an object of $\calC$.
The \defi{Spanier--Whitehead} category $\calC[T^{\otimes -1}]$ is the category with objects $(C,n)$, where $C$ is an object of $\calC$ and $n \in \Z$, and morphisms given by 
\[
\Mor_{\calC[T^{\otimes -1}]}( (C,n), (C',n')) = \colim_{m \geq -n, -n'} \Mor_{\calC}(T^{\otimes m+n} \otimes C, T^{\otimes m+n'} \otimes C')
\] Composition of morphisms is given by the expected formula.
\end{df}

By \cite[Theorem 4.3]{Voevodsky-ICMA1}, the composition $(C,n) \otimes (C',n') := (C \otimes C', n+n')$ defines a symmetric monoidal structure on $\calC[T^{\otimes -1}]$  provided that the cyclic permutation of $T \otimes T \otimes T$ is the identity. Moreover, in this case, the functor $\calC \to \calC[T^{\otimes -1}]$ sending $C$ to $(C,0)$ is symmetric monoidal and the object $(T,0)$ has tensor inverse $T^{\otimes -1} = (1_{\calC}, -1)$.

\begin{pr}\label{pr:SWfunctor}
Let $(\calC, \otimes_{\calC}, 1_{\calC}, \tau_{\calC})$ and $(\calD, \otimes_{\calD}, 1_{\calD}, \tau_{\calD})$ be a symmetric monoidal categories and let $T$ be an object of $\calC$ such that the cyclic permutation of $T \otimes T \otimes T$ is the identity. Suppose that $F: \calC \to \calD$ is a symmetric monoidal functor. Then there is a unique (up to unique isomorphism) symmetric monoidal functor $F[T^{\otimes -1}]: \calC[T^{\otimes -1}] \to \calD[F(T)^{\otimes -1}]$ such that the diagram
\begin{equation}\label{eqn:CtoDtoCT-1toDT-1}
\xymatrix{\calC \ar[r]^{F} \ar[d]& \calD \ar[d] \\
\calC[T^{\otimes -1}]\ar[r]_-{F[T^{\otimes -1}]} &\calD[F(T)^{\otimes -1}]}
\end{equation} commutes.
\end{pr}

\begin{proof}
Since the cyclic permutation of $T \otimes T \otimes T$ is the identity and $F$ is a symmetric monoidal functor, the cyclic permutation of $F(T) \otimes F(T) \otimes F(T)$ is the identity. Thus $\calD[F(T)^{\otimes -1}]$ is a symmetric monoidal category receiving a canonical symmetric monoidal functor from $\calD$.  Define $F[T^{\otimes -1}] (C, n):= (F(C),n)$. Then $F[T^{\otimes -1}]$ is a symmetric monoidal functor such that \eqref{eqn:CtoDtoCT-1toDT-1} commutes. Uniqueness follows from the isomorphism $(C,n) \cong C \otimes_{\calC[T^{\otimes -1}]} T^{\otimes_n}$ in $\calC[T^{\otimes -1}]$. Here $C$ and $T$ also denote their images $(C,0)$ and $(T,0)$, respectively, under the canonical functor $\calC \to \calC[T^{\otimes -1}]$. 
\end{proof}

\subsection{On the motivic Spanier--Whitehead category and Milnor--Witt $K$-theory} \label{sec:SWintro}

Let $\Sm_k$ denote the category of smooth schemes over a field $k$. We will be working with the Morel--Voevodsky $\mathbb{A}^1$-homotopy category $\calH(k)$ over $k$, and its pointed version $\calH(k)_{\ast}$ \cite[p.~109]{morelvoev}. A feature of $\Aone$-homotopy theory is two different analogues of the circle: $S^1$ and $\bA^1-\{0\}$. We use the following indexing convention for the resulting spheres: let $S^{1,0} = S^1$, $S^{1,1} = \bA^1-\{0\}$, and $S^{p,q}= (S^{1,0})^{\wedge (p-q)} \wedge  (S^{1,1})^{\wedge q}$. Let $\PP^1$ denote projective space over $k$ of dimension $1$ pointed at $\infty$. There is a weak equivalence $\PP^1 \simeq S^{2,1}$ in $\calH(k)_{\ast}$ induced from the pushout and homotopy pushout $\PP^1 \simeq \bA^1 \cup_{\G_m} \bA^1$. Inductive and gluing arguments show that $\bbb{A}^n-\{0\} \simeq S^{2n-1,n}$ and $\bbb{A}^n/ \bbb{A}^n-\{0\} \simeq S^{2n,n}$ \cite[Example 2.20]{morelvoev}. By \cite[Lemma 3.43(2)]{A1-alg-top} (for example), the cyclic permutation on $\PP^1 \wedge \PP^1 \wedge \PP^1$ is the identity in $\calH(k)_{\ast}$.
Let $$\cat{\SW}(k) := \calH(k)_{\ast}[(\PP^1)^{\otimes -1}]$$ be the Spanier--Whitehead category arising from $\calH(k)_{\ast}$ with chosen object $\PP^1$ (recall Definition~\ref{df:SWcat}); note that it is a symmetric monoidal category.

Given a vector bundle $V$ on a smooth scheme $X$, the associated \defi{Thom space} \cite[p.~110-114]{morelvoev} is defined as $\Th_X V := V/(V-0)$, where $V-0$ denotes the complement of the zero section. The stable $\mathbb{A}^1$-homotopy category $\SH(k)$ receives a symmetric monoidal functor from $\cat{\SW}(k)$ which is fully faithful on smooth schemes and their Thom spaces \cite[Theorem 5.2, Corollary 5.3]{Voevodsky-ICMA1}.

Fundamental theorems of Morel \cite{morel2004motivic-pi0,A1-alg-top} compute certain stable and unstable homotopy groups of spheres in terms of Milnor--Witt K-theory. Let $K^{\MW}_*(k) = \oplus_{i=-\infty}^{\infty} K^{\MW}_i(k)$ denote the \defi{Milnor--Witt $K$-theory} of a field $k$, defined by Morel and Hopkins to be the associative algebra generated by a symbol $\eta$ of degree $-1$ and symbols $[u]$ for $u \in k^*$ of degree $+1$, subject to the relations 
\begin{align*}
 [u][1-u] = 0, && [uv]= [u]+[v]+\eta[u][v], && \eta [u] = [u] \eta, && \eta h =0
\end{align*}
for all $u$,$v$ in $k^*$, where $h=2+[-1]\eta $ denotes the hyperbolic element. Let $\underline{K}^{\MW}_*$ denote the associated unramified sheaf on $\Sm_k$. (See \cite[Chapter 3.2]{A1-alg-top} for the definition and more information on the associated unramified sheaf.) Then Morel  \cite[Corollary 6.43]{A1-alg-top} shows that
\[
\Mor_{\SH(k)}(S^n, S^{n+m,m})  \cong \underline{K}^{\MW}_m\text{for all }n,m\quad \quad \Mor_{\calH(k)_{\ast}}(S^n, S^{n+m,m}) \cong \underline{K}^{\MW}_m, n \geq 2, m\geq 0.
\]

The $0$th graded piece $K^{\MW}_0(k)$ is isomorphic to the \defi{Grothendieck--Witt group} $\GW(k)$ \cite[Lemma 3.10]{A1-alg-top}, defined to be the group completion of the semi-ring of nondegenerate symmetric bilinear forms. There is a presentation of $\GW(k)$ with generators
$$\begin{array}{cccc}\langle a \rangle:& k\times k& \to& k \\
&(x,y)& \mapsto & axy\end{array}$$ for every $a\in k^{*},$ and relations given for all $a,b\in k^{*}$ by:
\begin{enumerate} \item $\langle ab^2\rangle  = \langle a \rangle$;
\item $\langle a \rangle + \langle b \rangle = \langle a + b \rangle + \langle ab(a+b) \rangle$
\item $\langle a \rangle \langle b \rangle = \langle ab \rangle$.
\end{enumerate} 
The sheaf $\underline{K}^{\MW}_0$, also denoted $\underline{\GW}$, is the Nisnevich sheaf associated to the presheaf sending a smooth $k$-scheme $Y$ to the group completion of the semi-ring of isomorphism classes of locally free sheaves $V$ on $Y$ equipped with a non-degenerate symmetric bilinear form $V \times V \to \calO_Y$.

Let $\Ab(k)$ denote the abelian category of Nisnevich sheaves of abelian groups on $\Sm_k$. For future reference, we record the following lemma, essentially due to Morel \cite[3]{A1-alg-top}.

\begin{lm}\label{HomAbA1_KMWnmn+m}
There exists a natural isomorphism
$ \underline{K}^{\MW}_m \to \underline{\Hom}_{\Ab(k)} (\underline{K}^{\MW}_n,\underline{K}^{\MW}_{n+m})$  for all $n>0$ and all $m$.
\end{lm}

\begin{proof}

Let $\Sheaves(\Sm_k, \Set_*)$ denote the $1$-category of Nisnevich sheaves of pointed sets on $\Sm_k$. By \cite[Theorem 3.36]{A1-alg-top}, the presheaf $(\G_m)^{\wedge n}: X \mapsto \cO^{\times}(X)^{\wedge n}$ is a sheaf of pointed sets.  By \cite[Theorem 3.37]{A1-alg-top} , $\underline{\Hom}_{\Ab(k)} (\underline{K}^{\MW}_n,\underline{K}^{\MW}_{n+m})\cong\underline{\Hom}_{\Sheaves(\Sm_k, \Set_*)}((\G_m)^{\wedge n}, \underline{K}^{\MW}_{n+m}) $. Let $(-)_{-1}$ denote the $-1$ construction of Voevodsky \cite[p 33]{A1-alg-top} on (pre)sheaves of groups on $\Sm_k$. For a sheaf of pointed sets $X$, let $\tilde{\ZZ}(X)$ denote the associated free sheaf of abelian groups (where the base point goes to $0$). It follows from the definition of $(-)_{-1}$, that for any $M$ in $\Ab(k)$ we have $M_{-1} \cong \underline{\Hom}_{\Ab(k)} (\tilde{\ZZ}(\G_m), M)$. Furthermore, applying $(-)_{-1}$ $n$-times, gives $M_{-n}\cong \underline{\Hom}_{\Ab(k)}  (\tilde{\ZZ}((\G_m)^{\wedge n}), M)$. Note that $\underline{\Hom}_{\Sheaves(\Sm_k, \Set_*)}((\G_m)^{\wedge n}, \underline{K}^{\MW}_{n+m}) \cong \underline{\Hom}_{\Ab(k)}  (\tilde{\ZZ}((\G_m)^{\wedge n}), M)$. Thus $\underline{\Hom}_{\Ab(k)} (\underline{K}^{\MW}_n,\underline{K}^{\MW}_{n+m}) \cong (\underline{K}^{\MW}_{n+m})_{-n} $. By \cite[Theorem 2.48]{A1-alg-top}, $(\underline{K}^{\MW}_{n+m})_{-n} \cong \underline{K}^{\MW}_{m}$.

\end{proof}

\subsection{$\Aone$-derived category and $\Aone$-homology}

Let $\Ch(\Ab(k))$ be the category of chain complexes $C_\ast$ of Nisnevich sheaves of abelian groups on $\Sm_k$ with differentials of degree $-1$. We denote by $D(\Ab(k))$ the associated derived category, obtained by inverting quasi-isomorphisms \cite[6.2]{A1-alg-top}. Let $\ZZ(\Aone)$ denote the free sheaf of abelian groups on the sheaf of sets represented by $\Aone$. Note that the map $\Aone \to \Spec k$ induces a map $\ZZ(\Aone) \to \ZZ$.

A chain complex $C_*$ in $\Ch(\Ab(k))$ is defined to be {\em $\Aone$-local} if for any $D_*$ in $\Ch(\Ab(k))$, the map
\[
\Hom_{D(\Ab(k))}(D_*, C_*) \to \Hom_{D(\Ab(k))}(D_* \otimes \ZZ(\Aone), C_*)
\] is a bijection \cite[Definition 6.17]{A1-alg-top}. The $\Aone$-derived category $D_{\Aone}(\Ab(k))$ \cite[Definition 6.17]{A1-alg-top} is obtained from $\Ch(\Ab(k))$ by inverting the $\Aone$-quasi-isomorphisms, defined to be morphisms $f: C_{\ast} \to D_{\ast}$ such that for all $\Aone$-local chain complexes $E_{\ast}$
\[
\Hom_{D(\Ab)}(D_{\ast}, E_{\ast}) \to \Hom_{D(\Ab)}(C_{\ast}, E_{\ast})
\] is bijective. There is an {\em (abelian) $\Aone$-localization functor} $\LaoneAb: \Ch(\Ab(k)) \to \Ch(\Ab(k))$ inducing a left-adjoint to the inclusion of $\Aone$-local complexes in $D(\Ab(k))$ \cite[6.18, 6.19]{A1-alg-top}; it induces an equivalence of categories between $D_{\Aone}(\Ab(k))$ and the full subcategory of $\Aone$-local complexes.

For a simplicial Nisnevich sheaf $\calX$ on $\Sm_k$, let $C_\ast(\calX)$ in $\Ch(\Ab(k))$ denote the normalized chain complex associated to the free simplicial abelian group $\ZZ\calX$ on $\calX$. The $\Aone$-chain complex $C_\ast^{\Aone}(\calX)$ (see \cite[Chapter 6.2]{A1-alg-top}) is obtained by taking the $\Aone$-localization of $C_\ast(\calX)$. This defines a functor from the $\Aone$-homotopy category $\calH(k)$ to the derived category $D_{\Aone}(\Ab(k))$:
\[
C_\ast^{\Aone}: \calH(k) \to D_{\Aone}(\Ab(k)).
\] If $\calX$ is a pointed space, the reduced chain complex $\widetilde{C}_*(\calX)$, which is the kernel of the morphism of $C_\ast(\calX)$ to $\Z$, similarly gives rise to a reduced $\Aone$-chain complex $\widetilde{C}_*^{\Aone}(\calX)$.

\begin{df}
For any simplicial Nisnevich sheaf $\calX$ on $\Sm_k$ and integer $n$, the \defi{$n$th $\Aone$-homology sheaf} $H_n^{\Aone}(\calX)$ of $\calX$ is the $n$th homology sheaf of the $\Aone$-chain complex $C_\ast^{\Aone}(\calX)$. For a pointed space $\calX$, the \defi{$n$th reduced $\Aone$-homology sheaf} of $\calX$ is $\widetilde{H}_n^{\Aone} (\calX) := H_n^{\Aone} (\widetilde{C}_*^{\Aone}(\calX))$.
\end{df}

Morel shows that $\Aone$-homology sheaves are strictly $\Aone$-invariant \cite[Corollary 6.23]{A1-alg-top} in the following sense.

\begin{df}
Let $\calF$ be a sheaf of abelian groups on $\Sm_k$ for the Nisnevich topology.
Then $\calF$ is \defi{$\Aone$-invariant} if for every $U \in \Sm_k$, the projection map $U \times \Aone \to U$ induces a bijection $\calF(U) \to \calF(U \times \A^1)$.
The sheaf $\calF$ is said to be \defi{strictly $\Aone$-invariant} if for every $U \in \Sm_k$ and every integer $i \geq 0$, the projection map $U \times \Aone \to U$ induces a bijection 
$$H^i_{Nis}(U,\calF) \to H^i_{Nis}(U \times \Aone,\calF).$$
The category of strictly $\Aone$-invariant sheaves on $\Sm_k$  is denoted $\Ab_{\Aone}(k)$.
\end{df}

For example, for a presheaf $\calX$ of sets on $\Sm_k$, the $0$th $\Aone$-homology $H_0^{\Aone}(\calX)$ is the free strictly $\Aone$-invariant sheaf on $\calX$. Morel computes 
\begin{equation}\label{HnAn/An-0}
\widetilde{H}_n^{\Aone} ( \bbb{A}^n/ \bbb{A}^n-\{0\}) \cong \underline{K}^{\MW}_n \text{ for }n \geq 1.
\end{equation}  This follows from the $\Aone$-weak equivalence $\bbb{A}^n/ \bbb{A}^n-\{0\} \simeq (S^1)^{\wedge n} \wedge (\G_m)^{\wedge n}$ \cite[p.~110, Spheres, Suspensions, Thom Spaces]{morelvoev}, the suspension isomorphism $\widetilde{H}_n^{\Aone} ((S^1)^{\wedge n} \wedge (\G_m)^{\wedge n}) \cong \widetilde{H}_{n-1}^{\Aone} ((S^1)^{\wedge n-1} \wedge (\G_m)^{\wedge n})$ \cite[Remark 6.30]{A1-alg-top} and  $\widetilde{H}_{0}^{\Aone}( (\G_m)^{\wedge n}) \cong  \underline{K}^{\MW}_n$\cite[Theorem 3.37 and Theorem 5.46]{A1-alg-top}.  For $n=0$, because $\bbb{A}^n/ \bbb{A}^n-\{0\} \cong \Spec k_+$ and $\widetilde{H}_n^{\Aone}(\Spec k_+) \cong \ZZ$, we have
\[
\widetilde{H}_n^{\Aone} ( \bbb{A}^n/ \bbb{A}^n-\{0\}) \cong \ZZ \text { for } n=0. 
\]
 
By \cite[Corollary 6.24]{A1-alg-top}, $\Ab_{\Aone}(k)$ is an abelian category and the inclusion $\Ab_{\Aone}(k) \subset \Ab(k)$ is exact. There is a tensor product, denoted $\otimes_{\Aone}$ or $ \otimes_{\Ab_{\Aone}(k)}$, on $\Ab_{\Aone}(k)$,  defined by $M \otimes_{\Aone} N := \pi_0 \LaoneAb( M \otimes_{\Ab(k)}N)$ for $N,M \in \Ab_{\Aone}(k)$. The map $M \otimes N \to M \otimes_{\Aone} N$ is the initial map to a strictly $\Aone$-invariant sheaf, as we now explain: by \cite[Theorem 6.22]{A1-alg-top},  $\LaoneAb( M \otimes_{\Ab(k)}N)$ is $-1$-connected, and by \cite[Corollary 6.23]{A1-alg-top}, $\pi_0 \LaoneAb( M \otimes_{\Ab(k)}N)$ is an element of $\Ab_{\Aone}(k)$. A strictly $\Aone$-invariant sheaf $F$, by \cite[Corollary 6.23]{A1-alg-top}, is $\Aone$-local. Then a map $M \otimes N \to F$ factors uniquely $M \otimes N  \to \LaoneAb (M \otimes N )\to F$. Since $F$ has no higher homotopy groups, we have that $M \otimes N \to F$ thus factors uniquely $M \otimes N  \to \LaoneAb (M \otimes N )\to \pi_0 \LaoneAb (M \otimes N) \to F$.

\begin{df}
(See \cite[Definition 2.9]{MorelSawant}.) A scheme $X$ in $\Sm_k$ is \defi{cohomologically trivial} if for every $n \geq 1$ and every strictly $\Aone$-invariant sheaf $M \in \Ab_{\Aone}(k)$, we have $H^n_{\nis}(X,M)=0$.  
\end{df}

If $X$ in $\Sm_k$ is cohomologically trivial, then $H_0^{\Aone}(X)$ is projective in $\Ab_{\Aone}(k)$ \cite[Example 2.36]{MorelSawant}. Moreover, if $X$ and $Y$ are cohomologically trivial in $\Sm_k$, so is $X \times Y$, and we have $H_0^{\Aone}(X \times Y) \cong H_0^{\Aone}(X) \otimes_{\Ab_{\Aone}(k)} H_0^{\Aone}(Y)$ \cite[Remark 2.10 and p.~12]{MorelSawant}.

Let $D(\Ab_{\Aone}(k))$ denote the bounded derived category of $\Ab_{\Aone}(k)$, obtained from the category $\Ch(\Ab_{\Aone}(k))$ by inverting quasi-isomorphisms. Let $D(\Ab^p_{\Aone}(k))\subseteq D(\Ab_{\Aone}(k))$ denote the full subcategory on bounded complexes $C_{\ast}$ of strictly $\Aone$-invariant sheaves $C_\ast$ that are isomorphic either to $H_0^{\Aone}(X)$  for some cohomologically trivial $X$ in $\Sm_k$ or to $\underline{K}^{\MW}_n \otimes_{\Ab_{\Aone}(k)} H_0^{\Aone}(X)$ for some cohomologically trivial $X$ and $n\geq 1$. In particular, the $C_\ast$ are projective \cite[Remark 2.26(1)]{MorelSawant}. 

We give $D(\Ab^p_{\Aone}(k))$ the structure of a symmetric monoidal category as follows. For chain complexes $P_{\ast}$ and $P'_{\ast}$ representing objects of $D(\Ab^p_{\Aone}(k))$, the set of homomorphisms $\Hom_{D(\Ab^p_{\Aone}(k))}(P_{\ast}, P'_{\ast})$ is given by homotopy classes of chain maps. Moreover, because $\underline{K}^{\MW}_n \otimes_{\Ab_{\Aone}(k)} \underline{K}^{\MW}_m \cong \underline{K}^{\MW}_{n+m}$ for $n,m \geq 1$ \cite[Theorem 3.37]{A1-alg-top} and $H_0^{\Aone}(X \times Y) \cong H_0^{\Aone}(X) \otimes_{\Ab_{\Aone}(k)} H_0^{\Aone}(Y)$ as above, the tensor product of bounded chain complexes in $\Ch(\Ab_{\Aone}(k))$ determines a well-defined derived tensor product

\begin{align}
\otimes_{D(\Ab^p_{\Aone}(k))}  : D(\Ab^p_{\Aone}(k)) \times D(\Ab^p_{\Aone}(k)) &\to D(\Ab^p_{\Aone}(k)) \label{eq:otimesDAbAone} \\
 P_{\ast} \otimes_{D(\Ab_{\Aone}(k))} P'_{\ast} &:= P_{\ast} \otimes_{\Ch(\Ab_{\Aone})(k)} P'_{\ast} \nonumber \\ 
\textrm{with } \qquad d(x \otimes y) &= d \otimes 1 + (-1)^{\vert x \vert} 1 \otimes d. \nonumber
\end{align}

The symmetry isomorphism is defined as
\begin{equation}\label{eq:tau_DAbAone}
\tau_{D(\Ab_{\Aone}(k))}: P_{\ast} \otimes_{D(\Ab^p_{\Aone}(k))} P'_{\ast}  \to P'_{\ast} \otimes_{D(\Ab^p_{\Aone}(k))} P_{\ast} 
\end{equation} so that its restriction to $P_i \otimes P'_j$ maps to $P'_j \otimes P_i$ by multiplication by $( -1 )^{ij}$ composed with the swap isomorphism for $\otimes_{\Ab_{\Aone}(k)}$. The swap map \eqref{eq:tau_DAbAone} and the derived tensor product of chain complexes \eqref{eq:otimesDAbAone} give $D(\Ab^p_{\Aone}(k))$ the structure of a symmetric monoidal category.

\section{$\Aone$-Spanier--Whitehead category of cellular smooth schemes}\label{Aone-Spanier-Whitehead-cellular}

In this section, we define cellular schemes, which are closely related to the classical notion of schemes with an affine stratification, and we introduce the cellular Spanier-Whitehead category that we will use for the rest of the paper.

\subsection{Cellular schemes} \label{sec:cellularschemes}

There are numerous useful definitions of a cellular scheme or cellular object of $\SH(k)$ available in the literature; see, e.g.,\cite[Example 1.9.1]{fulton-intersection}, \cite[Definition 1.16]{EisenbudHarris} \cite[Section 3]{Totaro-linear-varieties}, \cite[Def 3.4.1]{RozenblyumST}, \cite[Definitions 2.1, 2.7, and 2.10]{Dugger-Isaksen-Motivic_cell_structures}, \cite[Definitions 2.7 and 2.11]{MorelSawant}, and \cite[Question 1]{Asok_Fasel_Hopkins_alg_vector_bundles}. We will use \cite[Definition 2.11]{MorelSawant} and a slight modification adapted to our purposes.

The basic goal is to define a class of varieties that are built out of simple varieties (from a cohomological perspective) in such a way that their cohomology can be understood inductively.  A first natural class of such varieties are those built out of affine spaces.  

\begin{df}\label{df:strict_cellular_structure}
An \(n\)-dimensional smooth scheme \(X\) has a \defi{strict cellular structure} if one of the following equivalent conditions holds:
\begin{enumerate}
\item(cf.~\cite[Example 1.9.1]{fulton-intersection})\label{cond:F} \(X\) admits a filtration by closed subschemes
\[\emptyset  = \Sigma_{-1}(X) \subset \Sigma_0(X) \subset \Sigma_1(X) \subset \cdots \subset \Sigma_{n-1}(X) \subset \Sigma_n(X) = X,\]
such that \(\Sigma_i(X) \smallsetminus \Sigma_{i-1}(X)\) is a (finite) disjoint union of schemes isomorphic to affine space \(\bA^{i}\).
\item(cf.~\cite[Definition 2.7]{MorelSawant}) \label{cond:MW} \(X\) admits a filtration by open subschemes
\[\emptyset = \Omega_{-1}(X) \subset \Omega_0(X) \subset \Omega_1(X) \subset \cdots \subset \Omega_{n-1}(X) \subset \Omega_n(X) = X,\]
such that \(\Omega_i(X) \smallsetminus \Omega_{i-1}(X)\) is a (finite) disjoint union of schemes isomorphic to affine space \(\bA^{n-i}\).
\end{enumerate}
To pass between these conditions, let \(\Omega_i(X) \colonequals X \smallsetminus \Sigma_{n-i-1}(X) \) and \(\Sigma_i(X) \colonequals X \smallsetminus \Omega_{n-i-1}(X) \).
A scheme with a strict cellular structure is called a \defi{strict cellular scheme}.
\end{df} 

A related notion is an \defi{affine stratification} \cite[Definition 1.16]{EisenbudHarris}, where
\(X\) is a disjoint union of finitely many locally closed subschemes \(U_i\) (``\defi{open strata}''), each isomorphic to \(\bA^{n_i}\), whose closures \(\overline{U}_i\) are a disjoint union of strata \(U_j\).
If \(X\) has an affine stratification, then defining \(\Sigma_j(X)\) to be the union of all open strata \(U_i\) of dimension at most \(j\) yields a strict cellular structure.

\begin{ex}
Projective space \(\PP^n\) has a strict cellular structure given by a full flag \(\PP^0 \subset \PP^1 \subset \cdots \subset \PP^n\) with \(\Sigma_i(X) \colonequals \PP^{i}\) and \(\Omega_i(X) \colonequals \PP^n \smallsetminus \PP^{n-i-1}\).  Because the join $\PP^i \ast \PP^j $ is weakly equivalent to $ \PP^{i+j+1}$, i.e.  \(\PP^i \ast \PP^j \simeq_{\bA^1} \PP^{i+j+1}\), and for topological spaces $Y,Z$, the open subset $Y \ast Z - Y$ deformation retracts onto $Z$, it is useful to think of $\Omega_i(X)$ as an open neighborhood of the $i$-skeleton. See \cite[Remark 2.12]{MorelSawant} for more discussion of the topological notion of skeleta, CW-structure, and cellular structure.

\end{ex}

\begin{ex}\label{B/G_example}
The class of schemes with a strict cellular structure contains many geometrically interesting examples.
Grassmannians \(\mathbb{G}(r, n)\) have an affine stratification by the Schubert cells \cite[Section 4.1.2]{EisenbudHarris}.  More generally, the Bruhat decomposition on the flag vareity $G/B$ of a split semisimple algebraic group $G$ with a Borel subgroup $B$ defines a strict cellular structure \cite[p.~12]{MorelSawant}. 
\end{ex}

The following more general definition of Morel--Sawant allows the ``cells'' to be more general affine schemes. 

\begin{df}\label{df:cellular_structure}
(\cite[Definition 2.11]{MorelSawant}) Let $X$ be a smooth scheme over a field $k$. A \defi{cellular structure} on $X$ consists of an increasing filtration 
\[
\emptyset = \Omega_{-1}(X) \subsetneq \Omega_0(X) \subsetneq \Omega_{1}(X) \subsetneq \cdots \subsetneq \Omega_s(X) =X 
\] by open subschemes such that for $i \in \{0,\ldots,s \}$, the reduced induced closed subscheme $X_i := \Omega_i(X) \smallsetminus \Omega_{i-1}(X)$ of \(\Omega_i(X)\) is $k$-smooth, affine, everywhere of codimension $i$ and cohomologically trivial. A scheme with a cellular structure is called a \defi{cellular scheme}.
\end{df}

\begin{df}\label{df:finite_cellular_structure}
A \defi{cellular structure} on $X$ is called \defi{simple} if for every $n$, $H^{\Aone}_0(X_n) \simeq \ZZ^{b_n}$ for some non-negative integer~$b_n$. A scheme with a simple cellular structure is called a \defi{simple cellular scheme}.
\end{df}

The class of simple cellular schemes contains all schemes with a strict cellular structure (which in turn contains all schemes with an affine stratification), because \(H^{\Aone}_0(\bA^i) \simeq \ZZ\) and \(H^n_{\nis}(\bA^i,M)=0\) for all \(n > 0\) and any strictly $\Aone$-invariant sheaf $M \in \Ab_{\Aone}(k)$.

\subsection{Cellular Spanier--Whitehead category}\label{subsection:Cellular_SW_cat}

Let $\calH(k)^{\cell}_{\ast}$ denote the full subcategory of $\calH(k)_{\ast}$ with objects consisting of pointed spaces of the form $\Th_X V$, where $X$ is a smooth $k$-scheme admitting a cellular structure and $V$ is a vector bundle on $X$. Note that $X_{+} : = X \coprod \Spec k$ is an object of $\calH(k)^{\cell}_{\ast}$ as $X_+ \cong \Th_{X} 0$. The product of two smooth schemes admitting cellular structures admits a canonical cellular structure (see \cite[p.21]{MorelSawant} or Lemma~\ref{C:Kunnethpos}). It follows that
the symmetric monoidal structure $\wedge$ on $\calH(k)_{\ast}$ restricts to a symmetric monoidal structure on $\calH(k)^{\cell}_{\ast}$.

We now introduce the cellular Spanier--Whitehead category 
$$\cat{\SW}^{\cell}(k):= \calH(k)^{\cell}_{\ast}[(\PP^1)^{\otimes -1}].$$
Recall that we defined $\PP^1$ to be the projective line pointed at $\infty$. The pointed space $\PP^1$ is in $\calH(k)^{\cell}_{\ast}$ because $\PP^1 = \Th_{\Spec k} \calO$, so $\cat{\SW}^{\cell}(k)$ is well-defined. Note that the canonical functor $\cat{\SW}^{\cell}(k) \to \cat{\SW}(k)$ is fully faithful. As noted in Section \ref{S:categorical}, since the cyclic permutation on $\PP^1$ is the identity, the natural functor $\calH(k)^{\cell}_{\ast} \to \cat{\SW}^{\cell}(k)$ is a symmetric monoidal functor.

There is a canonical $\Aone$-weak equivalence
\[
\Th_X V \stackrel{\sim}{\leftarrow} \PP(V \oplus \calO)/\PP(V),
\] between the Thom space $\Th_X V := V/(V - 0)$ and the quotient $ \PP(V \oplus \calO)/\PP(V)$, where $\calO$ denotes the trivial bundle. It follows that there is a canonical isomorphism $\Th_X (V \oplus \calO) \cong \PP^1 \wedge \Th_X V$ in $\calH(k)_*$ \cite[Prop 2.17]{morelvoev}. This can be used to extend the Thom space construction to a functor
\[
\Th: K_0(X) \to {\cat{\SW}(k)} \quad \quad \Th V := (\Th_X (V \oplus \calO^{r}), -r)
\] 
where $r$ denotes a positive integer such that $V \oplus \calO^r$ is represented by a vector bundle. Indeed, a path in the K-theory groupoid $K(X)$ between vector bundles $V$ and $W$ produces a canonical isomorphism $\Th_X V \simeq \Th_X W$ in $\SH(k)$ \cite[Scholie~1.4.2(2)]{ayoub2007six}. This isomorphism lies in $\SW(k)$ because the functor $\SW(k) \to \SH(k)$ is fully faithful on Thom spaces \cite[Corollary 5.3]{Voevodsky-ICMA1}. For $V$ a vector bundle (of non-negative rank), there is a canonical isomorphism between the Thom space functor $\Th$ just defined applied to $V$ and the usual $\Th_X V$, justifying the abuse of notation. When $X$ admits a simple cellular structure, the Thom space functor  $\Th$ factors through $\cat{\SW}^{\cell}(k) \to \cat{\SW}(k)$, defining $\Th: K_0(X) \to {\cat{\SW}^{\cell}(k)}$. 

\begin{lm}\label{Xsmproj_is_dualizable_SW}
Let $X$ be a smooth projective scheme over $k$ and suppose that $X$ admits a cellular structure. Then $X_+$ is dualizable in $\cat{\SW}^{\cell}(k)$.
\end{lm}

\begin{proof}
Since $X$ is smooth and projective over $k$, we have that $X_+$ is dualizable in $\SH(k)$ with dual $\D X_+ \simeq \Th(-TX)$, where $TX$ denotes the tangent bundle of $X$ \cite[Th\'eor\`eme 2.2]{riou2005dualite}. Let $\D X_+$ in $\cat{\SW}^{\cell}(k)$ denote $\Th(-TX)$ by a slight abuse of notation. The evaluation and coevaluation maps $\epsilon$ and $\eta$ in $\SH(k)$ of Definition \ref{df:dualizable_object_symmetric_monoidal_cat} have unique preimages under $\cat{\SW}^{\cell}(k) \to \cat{\SW}(k) \to \SH(k)$ because both functors are fully faithful on  $1_{\cat{\SW}^{\cell}(k)}$, $X_+$, $\D X$, $X_+ \wedge \D X$, etc. 
\end{proof}

\section{The cellular homology of Morel--Sawant on cellular Thom spaces} \label{sec:cellularhomology}

Morel and Sawant \cite[\S 2.3]{MorelSawant} define the \defi{cellular $\Aone$-chain complex} and corresponding \defi{cellular $\Aone$-homology} for a cellular scheme, and show these are functors between appropriate categories \cite[Corollary 2.43]{MorelSawant} preserving monoidal structures \cite[Lemma 2.31]{MorelSawant}. We generalize these definitions in a straightforward manner to include Thom spaces, which will result in a symmetric monoidal functor
\[
\tilde{C}_{\ast}^{\cell}:  \calH(k)^{\cell}_{\ast} \to D(\Ab^p_{\Aone}(k))
\]
(see Proposition~\ref{Ccell_symmetric_monoidal}).

Let $V$ be a vector bundle of rank $r \geq 0$ on $X$ in $\Sm_k$ and
\[
 \emptyset = \Omega_{-1} \subsetneq \Omega_0 \subsetneq \Omega_{1} \subsetneq \cdots \subsetneq \Omega_s =X 
\] be a cellular structure on $X$, so $\Th_X V $ is an object of $ \cat{\SW}^{\cell}(k)$. Let $X_i = \Omega_i \setminus \Omega_{i-1} = \coprod_{m \in M_i} X_{im}$ be the decomposition of $X_i$ into connected components. Morel--Voevodsky Purity \cite[Theorem 2.23]{morelvoev} provides a canonical weak equivalence 
\[
\Th_{\Omega_i}V/\Th_{\Omega_{i-1}}V \simeq_{\Aone} \Th_{X_i}(V+\nu_i)
\] where $\nu_i$ is the normal bundle of the closed immersion $X_i \hookrightarrow \Omega_i$ of smooth $k$-schemes. 

Consider the cofibration sequence
\[
\Th_{\Omega_{i-1}} V \to \Th_{\Omega_i} V \to \Th_{\Omega_i} V /\Th_{\Omega_{i-1}} V,
\]
and its long exact sequence of reduced $\Aone$-homology sheaves
\[
\xymatrixcolsep{1pc}
\xymatrixrowsep{1pc}
\xymatrix{
\cdots \ar[r]  & \widetilde{H}^{\Aone}_{n} (\Th_{\Omega_{i-1}} V) \ar[r] & \widetilde{H}^{\Aone}_{n} (\Th_{\Omega_{i}} V) \ar[r]&  \widetilde{H}^{\Aone}_{n} (\Th_{\Omega_{i}} V / \Th_{\Omega_{i-1}} V) \ar[r] \ar[d]^{\simeq} & \widetilde{H}^{\Aone}_{n-1} (\Th_{\Omega_{i-1}} V) \ar[r] & \cdots  \\
& & &  \widetilde{H}^{\Aone}_{n} (\Th_{X_i} (V+\nu_i))
}
\]
which gives the boundary map $\partial_n$ as the composite
\[
 \widetilde{H}^{\Aone}_{n+r}(\Th_{X_n}(V+\nu_n)) \to   \widetilde{H}^{\Aone}_{n+r-1} (\Th_{\Omega_{n-1}} V) \to  \widetilde{H}^{\Aone}_{n+r-1} (\Th_{\Omega_{n}} V) \to  \widetilde{H}^{\Aone}_{n+r-1} (\Th_{X_{n-1}} (V+\nu_{n-1})).
\]
For $n\geq 0$, let
\[
\tilde{C}_{n+r}^{\cell} (\Th_X V) := \tilde{C}_{n+r}^{\cell} (X,V) :=  \widetilde{H}^{\Aone}_{n+r} (\Th_{X_n}(V+ \nu_n))
\]
with boundary maps $\partial_n$ as above be the \defi{(shifted, reduced) cellular $\Aone$-chain complex} on $(X,V)$. With this notation, we record the cellular structure of $X$ also when we write $\Th_X V$. Note also that $\Th_X 0 \simeq X_+$ and $\tilde{C}_{\ast}^{\cell} \Th_X 0 \simeq C^{\cell}_{\ast} X$ with $C^{\cell}_{\ast} X$ defined by Morel and Sawant \cite[Section 2.3]{MorelSawant}.

\begin{rmk}\label{rmk:finite_cellular_schemes_Ccell_KMWbi}
Suppose $X$ in $\Sm_k$ is equipped with a simple cellular structure. By definition, $\tilde{C}_{n}^{\cell}(\Th_X V) \cong  \widetilde{H}^{\Aone}_{n} (\Th_{X_{n-r}}(V+ \nu_n))$. By  \cite[Lemma 2.13]{MorelSawant}, we may choose a trivialization of $V+ \nu_n$. Such a choice defines an $\Aone$-weak equivalence $\Th_{X_{n-r}}(V+ \nu_n) \simeq (X_{n-r})_+ \wedge (\bbb{A}^n/ \bbb{A}^n-\{0\})$. By Morel's computation~\eqref{HnAn/An-0}, we have $ \widetilde{H}^{\Aone}_{n} ((X_{n-r})_+ \wedge \bbb{A}^n/ \bbb{A}^n-\{0\})) \cong H^{\Aone}_0(X_{n-r}) \otimes \underline{K}^{\MW}_n $ for $n \geq 1$ and $\widetilde{H}^{\Aone}_{n} ((X_{n-r})_+ \wedge \bbb{A}^n/ \bbb{A}^n-\{0\})) \cong   H^{\Aone}_0(X_{n-r})$ for $n=0$. Since the cellular structure on $X$ is simple,  $H^{\Aone}_0(X_{i}) \cong \ZZ^{b_i}$. Thus we have 
\[
 \tilde{C}^\cell_i(\Th_X V)\cong \begin{cases}
(\underline{K}^{\MW}_i)^{b_{i-r}}   & i>0 \\
 \ZZ^{b_{-r}} & i=0.
\end{cases}
\]
If the cellular structure is additionally strict, then $\Omega_i \setminus \Omega_{i-1} \cong \coprod_{j=1}^{b_i} \bbb{A}^{n-i}$. In other words, the integer $b_i$ corresponds to the number of connected components of $\Omega_i \setminus \Omega_{i-1}$.
\end{rmk}

\begin{rmk}
More generally, for $X$ in $\Sm_k$ equipped with a cellular structure, $\tilde{C}_n^{\cell} (\Th_X V) \cong  H^{\Aone}_0(X_{n-r}) \otimes \underline{K}^{\MW}_n $. Thus $\tilde{C}_{\ast}^{\cell} (\Th_X V) $ is in $D(\Ab_{\Aone}^p(k))$.

\end{rmk}

To show functoriality of $\tilde{C}_{\ast}^{\cell}$, Morel and Sawant introduce the notion of a strict $\Aone$-resolution \cite[Definition 2.33]{MorelSawant} and we will need this notion as well.

\begin{pr}\label{pr:pos_shifted_cell_A1-solvable}
Let $\Th_X V$ be an object of $\calH_{\ast}^{\cell}(k)$. There exists a unique morphism
\[
\phi_X: \tilde{C}_*^{\Aone}(\Th_X V) \to \tilde{C}_*^{\cell}(\Th_X V)
\] in $D(\Ab_{\Aone}(k))$ that is a strict $\Aone$-resolution in the sense  that the functor
\[
D(\Ab_{\Aone}(k)) \to \Ab \quad \quad C_* \mapsto \Hom_{D_{\Aone}(k)}(\tilde{C}_*^{\Aone}(\Th_X V), C_*)
\] is represented by $\tilde{C}_*^{\cell}(\Th_X V)$ .
\end{pr}

\begin{proof}
The proof of \cite[Proposition 2.37]{MorelSawant} extends to this level of generality.
\end{proof}

\begin{co}\label{f:map_pos_cellular_to_map_cCell}
Let $f: \Th_X V \to \Th_Y W$ be a morphism in $\calH(k)_{\ast}^{\cell}$. Then there exists a canonical chain homotopy class of morphisms
\[
\tilde{C}_*^{\cell}(f): \tilde{C}_*^{\cell}( \Th_X V ) \to \tilde{C}_*^{\cell}( \Th_Y W)
\] in $\Hom_{\Ch_{\geq 0}(\Ab_{\Aone}(k))}(\tilde{C}_*^{\cell}( \Th_X V ), \tilde{C}_*^{\cell}(\Th_Y W) )$.
\end{co}

\begin{proof}
Since $\tilde{C}_*^{\Aone}(-): \calH_*(k) \to D_{\Aone}(k) $ is a functor (see \cite[p. 161]{A1-alg-top}), the morphism $f$ induces a map $\tilde{C}_*^{\Aone}(f): \tilde{C}_*^{\Aone}(\Th_X V) \to \tilde{C}_*^{\Aone}(\Th_Y W)$. Thus by Proposition~\ref{pr:pos_shifted_cell_A1-solvable}, there is a map  $\tilde{C}_*^{\cell}(f): C_*^{\cell}( \Th_X V ) \to \tilde{C}_*^{\cell}( \Th_Y W)$ in $D(\Ab_{\Aone}(k))$ such that the diagram
\[
\xymatrix{\tilde{C}_*^{\Aone}(\Th_X V) \ar[d]_{\phi} \ar^{\tilde{C}_*^{\Aone}(f)}[r] & \tilde{C}_*^{\Aone}(\Th_Y W)\ar[d]_{\phi}\\
\tilde{C}_*^{\cell}(\Th_X V) \ar[r]_{\tilde{C}_*^{\cell}(f)} & \tilde{C}_*^{\cell}(\Th_Y W)}
\] commutes. As noted previously, $\tilde{C}_*^{\cell}( \Th_X V )$ is a bounded complex of projective objects of $\Ab_{\Aone}(k)$. It follows that $\Hom_{D(\Ab_{\Aone}(k))}(\tilde{C}_*^{\cell}( \Th_X V ), \tilde{C}_*^{\cell}(\Th_Y W) )$ is the group of chain homotopy classes of maps in $\Hom_{\Ch_{\geq 0}(\Ab_{\Aone}(k))}(\tilde{C}_*^{\cell}( \Th_X V ), \tilde{C}_*^{\cell}(\Th_Y W) )$.
\end{proof}

The following is the generalization of \cite[Lemma 2.31]{MorelSawant} to include Thom spaces. The proof was omitted in \cite{MorelSawant}, so we include one for completeness.

\begin{lm}[K\"unneth Formula] \label{C:Kunnethpos}
Suppose that $X$ and $Y$ are smooth schemes equipped with cellular structures, and let $p_1$ and $p_2$ be the projections of $X \times_k Y$ to $X$ and $Y$, respectively. Let $V$ and $W $ be vector bundles on $X$ and $Y$, respectively. Then there exists a cellular structure on $X \times Y$ and a natural isomorphism in $D(\Ab^p_{\Aone}(k))$
\begin{equation} \label{eq:cellularkunnethpos}
\tilde{C}_\ast^{\cell}(\Th_X V) \otimes_{D(\Ab^p_{\Aone}(k))} \tilde{C}_\ast^{\cell}(\Th_Y W) \to \tilde{C}_\ast^{\cell}(\Th_{X \times Y}(p_1^* V + p_2^*W)).
\end{equation}
\end{lm}

\begin{proof}
Let $r$ and $s$ be the (nonnegative) ranks of $V$ and $W$, respectively. Let $\nu_i$ and $\mu_i$ denote the normal bundles of $X_i = \Omega_i(X) \setminus \Omega_{i-1}(X) \hookrightarrow \Omega_i(X)$ and $Y_i = \Omega_i(Y) \setminus \Omega_{i-1}(Y) \hookrightarrow \Omega_i(Y)$, respectively.

Note that there is a natural shifted cellular structure on $X \times Y$ equipped with $p_1^* V + p_2^* W \in K_0(X \times Y)$, where the open strata are $\Omega_n(X \times Y) = \cup_{i+j=n} (\Omega_i(X) \times \Omega_j(Y))$ and the closed strata are $(X \times Y)_n := \Omega_n(X \times Y) \setminus \Omega_{n-1}(X \times Y) = \sqcup_{i+j=n} X_i \times Y_j$. The normal bundle $\xi_n$ on the inclusion $(X \times Y)_n \hookrightarrow \Omega_n(X \times Y)$ is the disjoint union of the normal bundle on each $X_i \times Y_j$, namely $\xi_n = \sqcup_{i+j=n} (p_1^* \nu_i + p_2^* \mu_j)$.

We compare the degree $n+r+s$ terms of the two sides of \eqref{eq:cellularkunnethpos}. The degree $n+r+s$ term of the right hand side of \eqref{eq:cellularkunnethpos} is
\begin{align*}
\widetilde{H}_{n+r+s}^{\Aone}(\Th_{(X \times Y)_n}(p_1^*V& + p_2^* W + \xi_n)) \\
&\cong \widetilde{H}_{n+r+s}^{\Aone}(\Th_{\sqcup_{i+j=n} X_i \times Y_j}(p_1^*V + p_2^* W + \sqcup_{i+j=n} (p_1^* \nu_i + p_2^* \mu_j))) \\
&\cong \widetilde{H}_{n+r+s}^{\Aone}( \vee_{i+j=n} \Th_{X_i \times Y_j} (p_1^*V + p_2^*W + p_1^*\nu_i + p_2^* \mu_j)) \\
&\cong \bigoplus_{i+j=n} \widetilde{H}_{n+r+s}^{\Aone} (\Th_{X_i \times Y_j} (p_1^*V + p_2^*W + p_1^*\nu_i + p_2^* \mu_j)).
\end{align*}
The degree $n+r+s$ term of the left hand side of \eqref{eq:cellularkunnethpos} is
\[
\bigoplus_{i+j=n+r+s} \tilde{C}_i^\cell(\Th_X V) \otimes_{\Ab_{\Aone}(k)} \tilde{C}_j^\cell(\Th_Y W) \cong \bigoplus_{i+j=n} \widetilde{H}_{i+r}^{\Aone}(\Th_{X_i}(V+\nu_i)) \otimes_{\Ab_{\Aone}(k)} \widetilde{H}_{j+s}^{\Aone}(\Th_{Y_j}(W+\mu_j)).
\]

For any $X'$ and $Y'$ with $V' \in K_0(X')$ and $W' \in K_0(Y')$, we have a natural equivalence $\Th_{X'}(V') \times \Th_{Y'}(W') \to \Th_{X' \times Y'}(V' \times W')$. We therefore have an induced map \begin{align} \label{eq:kunnethoneterm}
\widetilde{H}_{i+r}^{\Aone}(\Th_{X_i}(V+\nu_i)) &\otimes_{\Ab_{\Aone}(k)} \widetilde{H}_{j+s}^{\Aone}(\Th_{Y_j}(W+\mu_j))\\
 &\to \widetilde{H}_{i+j+r+s}^{\Aone} (\Th_{X_i \times Y_j} (p_1^*V + p_2^*W +p_1^*\nu_i + p_2^* \mu_j)). \nonumber
\end{align}
Unwinding the definitions and using the functoriality of the Thom space, \eqref{eq:kunnethoneterm} induces a map of chain complexes \eqref{eq:cellularkunnethpos}, i.e. is compatible with the differentials.

We claim that for $i+j = n$ the map \eqref{eq:kunnethoneterm} is an isomorphism. By \cite[Lemma 2.13]{MorelSawant}, both $V+\nu_i$ and $W+\mu_j$ represent trivial elements of $K_0(X_i)$ and $K_0(Y_j)$ respectively. Choosing trivializations induces a trivialization of $p_1^*V + p_2^*W +p_1^*\nu_i + p_2^* \mu_j$ on $X_i \times Y_j$. These trivializations and the suspension isomorphism for $\Aone$-homology \cite[Remark 6.30]{A1-alg-top} induce the following isomorphisms:
\[ 
\widetilde{H}_{i+r}^{\Aone}(\Th_{X_i}(V+\nu_i)) \cong \widetilde{H}_{i+r}^{\Aone}((X_i)_+ \wedge (\G_m)^{\wedge(i+r)} \wedge (S^1)^{\wedge(i+r)}) \cong \widetilde{H}_0^{\Aone}((X_i)_+ \wedge (\G_m)^{\wedge(i+r)} )
\]
\[
\widetilde{H}_{j+s}^{\Aone}(\Th_{Y_j}(W+\mu_j)) \cong \widetilde{H}_{j+s}^{\Aone}((Y_j)_+ \wedge (\G_m)^{\wedge(j+s)} \wedge (S^1)^{\wedge(j+s)}) \cong \widetilde{H}_0^{\Aone}((Y_j)_+ \wedge (\G_m)^{\wedge(j+s)} )
\]
\begin{align*}
\widetilde{H}_{i+j+r+s}^{\Aone}(\Th_{X_i \times Y_j}&(p_1^*V + p_2^*W +p_1^*\nu_i + p_2^* \mu_j)) \\
 &\cong \widetilde{H}_{i+j+r+s}^{\Aone}((X_i \times Y_j)_+ \wedge (\G_m)^{\wedge(i+j+r+s)} \wedge (S^1)^{\wedge(i+j+r+s)})\\
 &\cong \widetilde{H}_0^{\Aone}((X_i \times Y_j)_+\wedge (\G_m)^{\wedge(i+j+r+s)} ).
\end{align*}
Note that $\widetilde{H}_{0}^{\Aone}$ takes a sheaf of sets to the free strictly $\Aone$-invariant sheaf of abelian groups on the sheaf of sets, and thus transforms $\wedge$ to $\otimes$. The claim follows.
\end{proof}

\begin{lm}\label{C:symmetry}
Suppose that $X$ and $Y$ are smooth schemes equipped with cellular structures, and let $p_1$ and $p_2$ be the projections of $X \times_k Y$ to $X$ and $Y$, respectively. Let $V$ and $W$ be vector bundles on $X$ and $Y$, respectively. Then $\tilde{C}_*^{\cell}$ respects the symmetry maps $\tau_{\calH_{\ast}^{\cell}(k)}(\Th_X V, \Th_Y W)$ and $\tau_{D(\Ab^p_{\Aone}(k))}$ in the sense that the diagram 
\begin{equation}\label{cd:swapHCellDAnAnone}
\xymatrix{
 \tilde{C}_*^{\cell}(\Th_X V) \otimes \tilde{C}_*^{\cell}(\Th_Y W) \ar[r]^{\cong}_-{m_1} \ar[d]_{\tau_{D(\Ab^p_{\Aone}(k))} } &  \tilde{C}_*^{\cell}(\Th_{X \times Y}  p_1^* V + p_2^* W) \ar[d]^{\tilde{C}_*^{\cell}(\tau_{\calH_*^{\cell}})}
 \\
 \tilde{C}_*^{\cell}(\Th_Y W) \otimes \tilde{C}_{*}^{\cell}(\Th_X V) \ar[r]^{\cong}_-{m_2} & \tilde{C}_*^{\cell}(\Th_{Y\times X} p_2^* W+p_1^* V)
 \\
}\end{equation} commutes.
\end{lm}

\begin{proof}

With appropriate orientation data, \eqref{cd:swapHCellDAnAnone} becomes a sum of diagrams of the form
\begin{equation}\label{cd:oriented_swapHCellDAnAnone}
\xymatrix{
H^{\Aone}_0(X_{i-r}) \otimes_{\Aone} \underline{K}^{\MW}_i \otimes_{\Aone} H^{\Aone}_0(Y_{j-s}) \otimes_{\Aone} \underline{K}^{\MW}_j  \ar[r] \ar[d]_{\tau_{D(\Ab^p_{\Aone}(k))} } & H^{\Aone}_0(X_{i-r} \times Y_{j-s}) \otimes_{\Aone} \underline{K}^{\MW}_n \ar[d]^{\tilde{C}_*^{\cell}(\tau_{\calH_*^{\cell}})}
 \\
 H^{\Aone}_0(Y_{j-s}) \otimes_{\Aone} \underline{K}^{\MW}_j \otimes_{\Aone} H^{\Aone}_0(X_{i-r}) \otimes_{\Aone} \underline{K}^{\MW}_i \ar[r] & H^{\Aone}_0(Y_{j-s}\times X_{i-r}) \otimes_{\Aone} \underline{K}^{\MW}_n 
 \\
}
\end{equation} where $i+j = n$, $i,j \geq 1$, along with similar diagrams where $i$ or $j$ is $0$ and the appropriate factors of $ \underline{K}^{\MW}_{\ast} $ are omitted. In \eqref{cd:oriented_swapHCellDAnAnone} the map $\tilde{C}_*^{\cell}(\tau_{\calH_*^{\cell}})$ is induced by applying $\widetilde{H}^{\Aone}_{n}$ to a swap map on spaces of the form $(X_i)_+ \wedge S^{2i, i}\wedge (Y_{j})_+ \wedge S^{2j,j}$ with $X_i$ and $Y_j$ cohomologically trivial. This map is multiplication by the swap map $ S^{2i, i} \wedge S^{2j,j} \to S^{2j,j} \wedge S^{2i, i}$ in $\GW(k)$. This element of $\GW(k)$ equals $\langle -1 \rangle^{ij}$ by \cite[Lemma 3.43(2)]{A1-alg-top}, which states that $\tau(S^{p_1,q_1},S^{p_2,q_2})$ represents $(-1)^{(p_1-q_1)(p_2-q_2)}(- \langle -1 \rangle)^{q_1 q_2}$ in $\GW(k)$. By definition, the map $\tau_{D(\Ab^p_{\Aone}(k))}$ is multiplication by $(-1)^{ij}$ times the canonical swap on a tensor product (see \eqref{eq:tau_DAbAone}). 

The map $m_1$ is the tensor product of the isomorphism $H^{\Aone}_0(X_{i-r}) \otimes_{\Aone} H^{\Aone}_0(Y_{j-s}) \to H^{\Aone}_0(X_{i-r} \times Y_{j-s}) $ with the isomorphism $ \underline{K}^{\MW}_i \otimes_{\Aone}  \underline{K}^{\MW}_j \to  \underline{K}^{\MW}_n$ induced by multiplication on Milnor--Witt K-theory. A similar statement holds for $m_2$ where the factors are reversed. Since Milnor--Witt K-theory is $-\lra{-1}$ graded commutative \cite[Corollary 3.8]{A1-alg-top}, it follows that $m_2$ is a canonical swap on a tensor product to reorder the factors followed by $(-\lra{-1})^{ij}$ times $m_1$. Since $(-1)^{ij} (-\lra{-1})^{ij} = \langle -1 \rangle^{ij}$, the claim follows.
\end{proof}

\begin{pr}\label{Ccell_symmetric_monoidal}
The functor $\tilde{C}_{\ast}^{\cell}:  \calH(k)^{\cell}_{\ast} \to D(\Ab^p_{\Aone}(k))$ is symmetric monoidal.
\end{pr}

\begin{proof}
This follows from Corollary~\ref{f:map_pos_cellular_to_map_cCell}, Lemma~\ref{C:Kunnethpos}, and Lemma ~\ref{C:symmetry}.
\end{proof}

\section{Spanier--Whitehead cellular complex} \label{sec:SWcellcomplex}

\subsection{Definitions and basic properties}
Recall we defined $\cat{\SW}^{\cell}(k):= \calH(k)^{\cell}_{\ast}[(\PP^1)^{\otimes -1}]$ in Section~\ref{subsection:Cellular_SW_cat}. Since $\tilde{C}_*^{\cell}$ is a symmetric monoidal functor and the cyclic permulation of $\PP^1 \otimes \PP^1 \otimes \PP^1$ is the identity in $\calH(k)_{\ast}$, the cyclic permutation of $\tilde{C}_*^{\cell}(\PP^1) \otimes \tilde{C}_*^{\cell}(\PP^1) \otimes \tilde{C}_*^{\cell}(\PP^1)$ is the identity in $D(\Ab^p_{\Aone}(k))$. Denote the corresponding Spanier--Whitehead category (as in Definition \ref{df:SWcat}) by
$$\DpSWAone:=D(\Ab^p_{\Aone}(k))[\tilde{C}_*^{\cell}(\PP^1)^{\otimes -1}].$$
By \cite[Corollary 2.51]{MorelSawant}, the complex $C^{\cell}_{\ast}(\PP^1) \simeq \tilde{C}_*^{\cell}(\PP^1_+)$ is represented by the complex 
\[
\underline{K}^{\MW}_1 \stackrel{0}{\rightarrow} \ZZ. 
\] Thus $\tilde{C}^{\cell}_*(\PP^1) \simeq \underline{K}^{MW}_1[1]$ and we have
\begin{equation}\label{DpSWAone_KMW1invert}
\DpSWAone \simeq D(\Ab^p_{\Aone}(k))[(\underline{K}^{MW}_1[1])^{\otimes -1}].
\end{equation}

\begin{df}\label{df:CcellSW}
Define
\[
\CcellSW: \cat{\SW}^{\cell}(k)\to \DpSWAone
\] to be the symmetric monoidal functor obtained by applying Proposition~\ref{pr:SWfunctor} to the symmetric monoidal functor $\tilde{C}_*^{\cell}:  \calH(k)^{\cell}_{\ast} \to D(\Ab^p_{\Aone}(k))$ of Proposition \ref{Ccell_symmetric_monoidal}.
\end{df}

\begin{co}\label{cor:trace_commutes_with_Ccell}
Let $X$ be smooth and projective over a field $k$ and suppose that $X$ admits a cellular structure. For any endomorphism \(\varphi \colon X \to X\) and any integer \(m \geq 1\), we have
\[ \Tr(\CcellSW(\varphi^m)) = \CcellSW(\Tr(\varphi^m)). \] 
\end{co}

\begin{proof}
We may assume \(m=1\). By Proposition~\ref{Xsmproj_is_dualizable_SW}, $X$ is dualizable in $\cat{\SW}^{\cell}(k)$. The result follows by applying Proposition~\ref{HrespectDtr}  to the symmetric monoidal functor $\CcellSW$.
\end{proof}

\begin{pr}\label{L:changeofcategory}
The morphism
\begin{equation} \label{eq:CSWcellEnd} \CcellSW: \End(1_{\cat{\SW}^{\cell}(k)})  \to \End(1_{\DpSWAone})
\end{equation}
 is an isomorphism, and both sides are isomorphic to $\GW(k)$.
\end{pr}

\begin{proof}
We compute $\CcellSW(1_{\cat{\SW}^{\cell}(k)}) = \CcellSW(\Spec k_+) = \ZZ$.
By the identification \eqref{DpSWAone_KMW1invert}, we obtain 
\begin{align*}
\End&(1_{\DpSWAone}) \\
&= \colim_{n \geq 0} \End_{D(\Ab_{\Aone}(k))}(\underline{K}^{MW}_1[1]^{\otimes n}) \\
& \cong \colim_{n \geq 0} \End_{D(\Ab_{\Aone}(k))}(\underline{K}^{MW}_n[n]) &&\text{(by \cite[Theorem 3.37]{A1-alg-top})} \\
& \cong \colim_{n \geq 0} \End_{\Ab_{\Aone}(k)}(\underline{K}^{MW}_n) &&\text{(since $\underline{K}^{MW}_n$ is projective by \cite[Theorem 3.37]{A1-alg-top})} \\
 & \cong \colim_{n > 0} \GW (k) &&\text{(by Lemma~\ref{HomAbA1_KMWnmn+m})} \\
 &\cong \GW(k).
\end{align*}
We have $\End(1_{\cat{\SW}^{\cell}(k)}) =\colim_{n \geq 0} \End(S^{2n,n}) \cong \GW(k) $ by \cite[Cor 6.43]{A1-alg-top}. Since $H_*^{\Aone}$ respects the $\GW(k)$-module structure on $\End(S^{2n,n})$, so does $\CcellSW$. Thus the map \eqref{eq:CSWcellEnd} becomes the identity map on $\GW(k)$ under the given isomorphisms.
\end{proof}

\subsection{Endomorphisms, traces, and characteristic polynomials}\label{characteristic_poly_def}

Given an endomorphism \(\varphi\) of a dualizable object in \(\DpSWAone\), we can use the categorial trace of \(\varphi\) and its powers to define a \defi{logarithmic characteristic polynomial} of \(\varphi\), which, using Proposition~\ref{L:changeofcategory}, is a power series with coefficients in \(\GW(k)\).

\begin{df}\label{df:log_derivative_char_poly_varphi}
The \defi{logarithmic characteristic polynomial} of an endomorphism \(\varphi\) of a dualizable object in \(\DpSWAone\) 
is the power series defined by
\[\dlog P_\varphi(t) \colonequals \sum_{m=1}^\infty - \Tr(\varphi^m) t^{m-1} \in \GW(k)[[t]].\]
\end{df}

We use the convention that $\dlog$ indicates a formal logarithmic derivative, while $\frac{d}{dt}\log$ denotes the derivative of the logarithm. 

This definition is motivated by the usual definition of the characteristic polynomial of a square matrix.

\begin{df}\label{df:char_poly}
  For any commutative ring $R$ and endomorphism $\varphi$ of $R^n$ represented by a matrix $A$, the \defi{characteristic polynomial} \(P_\varphi(t)\) is defined as
\(P_\varphi(t) := \det(1 - At)\).  
\end{df}

We then have the following elementary relation.
\begin{lm}\label{lem:linear_algebra}
Let \(R\) be a commutative ring and let \(A \colon R^n \to R^n\) be an endomorphism.  Then
\[\frac{d}{dt} \log(P_A(t)) = \sum_{m=1}^\infty - \Tr(A^m)t^{m-1}.\]
\end{lm}
\begin{proof}
We will show that these two power series agree by comparing each coefficient.  The coefficient of \(t^i\) on each side is a polynomial (with integer coefficients) in the entries of the matrix representing the endomorphism \(A\).  To show that these integer polynomials agree, it suffices to prove that they take the same value on every complex number.  In other words, it suffices to prove the theorem when \(R = \CC\).  In this case, we may assume that the matrix of \(A\) is upper-triangular with diagonal entries \(c_1, \dots, c_n\) (since trace is independent of a choice of basis).  We have
\( P_A(t) = \prod_{i=1}^n (1-tc_i)\) and \(\Tr(A^m) = \sum_{i=1}^n c_i^m\),
which gives
\begin{align*}
\frac{d}{dt} \log(P_A(t)) &= \sum_{i=1}^n \left(- \sum_{m=1}^\infty c_i^m t^{m-1}\right) \\
&= \sum_{m=1}^\infty  \left(-\sum_{i=1}^nc_i^m\right) t^{m-1} \\
&= \sum_{m=1}^\infty  -\Tr(A^m) t^{m-1}. \qedhere
\end{align*}

\end{proof}

Let $C_{\ast}$ in $\Ch_{\geq 0}(\Ab_{\Aone}(k))$ be a bounded complex such that for $n>0$, we have $C_n \cong (\underline{K}^{MW}_n)^{b_n}$ for some nonnegative integers $b_n$ and $C_0 \cong  \ZZ^{b_{0}}$. For example, we may take $C_* = C^{\cell}_*(X)$ for a smooth simple cellular scheme $X$  (Remark~\ref{rmk:finite_cellular_schemes_Ccell_KMWbi}). Let $N$ be an integer. Let $t$ be a nonnegative integer such that $C_n = 0$ for $n \geq t$. Define 
\[ \Dual(C_{\ast}, N) \colonequals [\underline{\Hom}_{\Ch_{\geq 0}(\Ab(k))}(C_{\ast},\underline{K}^{MW}_{t}[t]),-t-N]. \]  
By Lemma~\ref{HomAbA1_KMWnmn+m}, $\Dual(C_{\ast}, N)$ is a representative for an object in $D^{\SW}(\Ab^p_{\Aone}(k))$.

\begin{pr}\label{P:dualinDSWAbAonek}
The object $\Dual(C_{\ast}, N)$ is a dual object to $(C_\ast, N)$ in $D^{\SW}(\Ab^p_{\Aone}(k))$ in the sense of Definition~\ref{df:dualizable_object_symmetric_monoidal_cat}. 
\end{pr}

\begin{proof}
Define $ \epsilon\colon  \Dual(C_\ast, N)  \otimes (C_\ast, N)\rightarrow (\ZZ,0)$ to be the map in $D^{\SW}(\Ab^p_{\Aone}(k))$ associated to the natural evaluation map 
 \[
 \epsilon[t]: \underline{\Hom}_{\Ab(k)}(C_{\ast},\underline{K}^{MW}_{t}[t]) \otimes C_\ast \to \underline{K}^{MW}_{t}[t] \cong \underline{K}^{MW}_{1}[1]^{\otimes t}
 \] in $\Ch_{\geq 0}(\Ab(k))$.

Define a map $ \underline{\Hom}_{\Ab(k)}(C_n,C_n) \otimes \underline{K}^{\MW}_t \to \underline{\Hom}_{\Ab(k)}(C_n,C_n \otimes_{\Aone} \underline{K}^{\MW}_t)$ by $\phi \otimes a \mapsto \phi_a$ where $\phi_a(c) = \phi(c) \otimes a$. Since $C_n \cong (\underline{K}^{MW}_n)^{b_n}$ or $n=0$ and $C_0 \cong \ZZ^{b_0}$, $\underline{K}^{MW}_n \otimes_{\Aone} \underline{K}^{\MW}_t \cong \underline{K}^{\MW}_{n+t}$ \cite[Theorem 3.37]{A1-alg-top}, and $\underline{\Hom}(\underline{K}^{\MW}_{n},\underline{K}^{\MW}_{n+t}) \cong \underline{K}^{\MW}_{t}$  (Lemma~\ref{HomAbA1_KMWnmn+m}). Thus $\underline{\Hom}_{\Ab(k)}(C_n,C_n \otimes_{\Aone} \underline{K}^{\MW}_t)$ is strictly $\Aone$-invariant. (Note that $t-n \geq 1$ by construction.) This defines a map 
\[
f_n: \underline{\Hom}_{\Ab(k)}(C_n,C_n) \otimes_{\Aone} \underline{K}^{\MW}_t \to \underline{\Hom}_{\Ab(k)}(C_n,C_n \otimes_{\Aone} \underline{K}^{\MW}_t)
\] 
 By similar reasoning, we have a map
 \[
 g_n: \underline{\Hom}_{\Ab(k)}(C_n, \underline{K}^{\MW}_t) \otimes_{\Aone} C_n  \to \underline{\Hom}_{\Ab(k)}(C_n,\underline{K}^{\MW}_t \otimes_{\Aone} C_n )
\] Using the isomorphisms $C_n \cong (\underline{K}^{MW}_n)^{b_n}$ with $n>0$ or $\ZZ^{b_0}$, $\underline{\Hom}(\underline{K}^{\MW}_{n},\underline{K}^{\MW}_{t}) \cong \underline{K}^{\MW}_{t-n}$, and $\underline{K}^{MW}_{t-n} \otimes_{\Aone} \underline{K}^{\MW}_n \cong \underline{K}^{\MW}_{t}$ (note that $t-n>0$ for us), we see that the map $g$ is an isomorphism. 
Let $$h_n:\underline{K}^{MW}_t \rightarrow  \underline{\Hom}_{\Ab(k)}(C_n,C_n) \otimes \underline{K}^{\MW}_t$$  denote the map defined by $a \mapsto 1_{C_n} \otimes a$. 

The composite $\tau(n,t-n) \circ g_n^{-1}\circ \tau(n,t) \circ f_n \circ h_n$ defines a map
\begin{align*}
\underline{K}^{MW}_t &\rightarrow  \underline{\Hom}_{\Ab(k)}(C_n,C_n) \otimes \underline{K}^{\MW}_t \\
&\to  \underline{\Hom}_{\Ab(k)}(C_n,C_n \otimes_{\Aone} \underline{K}^{\MW}_t) \stackrel{ (- \lra{-1})^{nt}}{\xrightarrow{\hspace*{2cm}}}  \underline{\Hom}_{\Ab(k)}(C_n, \underline{K}^{\MW}_t \otimes_{\Aone}  C_n)\\ 
&\to  \underline{\Hom}_{\Ab(k)}(C_n, \underline{K}^{\MW}_t) \otimes_{\Aone} C_n \stackrel{(- \lra{-1})^{n(t-n)}}{\xrightarrow{\hspace*{2cm}}} C_n \otimes_{\Aone}  \underline{\Hom}_{\Ab(k)}(C_n, \underline{K}^{\MW}_t) \end{align*} where $\tau(i,j)$ is the swap $a \otimes b \mapsto (- \lra{-1})^{ij} b \otimes a$. The sign comes from the graded $(- \lra{-1})$-commutativity of $ \underline{K}^{\MW}_\ast$. See \cite[Corollary 3.8]{A1-alg-top}.

Taking the product of the $\tau(n,t-n) \circ g_n^{-1}\circ \tau(n,t) \circ f_n \circ h_n$ over $n$ defines a map 
\[
 \eta[t]: \underline{K}^{MW}_{t}[t] \to C_\ast  \otimes_{\Ch(\Ab_{\Aone}(k))}   \underline{\Hom}(C_{\ast},\underline{K}^{MW}_{t}[t]).
 \] Let $\eta\colon (\ZZ,0) \rightarrow  (C_\ast, N) \otimes \Dual(C_\ast, N)$ be the associated map in $D^{\SW}(\Ab^p_{\Aone}(k))$. We leave checking that these maps satisfy the desired properties as in Definition~\ref{df:dualizable_object_symmetric_monoidal_cat} to the reader. \end{proof}

Let $C_{\ast}$ in $\Ch_{\geq 0}(\Ab_{\Aone}(k))$ be a bounded complex such that $C_i \cong (\underline{K}^{MW}_i)^{b_i}$ for   $i>0$ and $C_0 \cong \ZZ^{b_0}$ for some nonnegative integers $b_i$.
Let $\varphi: C_{\ast} \to C_{\ast}$ be a morphism in $D(\Ab_{\Aone}(k))$. As in the proof of Corollary~\ref{f:map_pos_cellular_to_map_cCell}, the morphism $\varphi$ is represented by a map of complexes, so we have maps $\varphi_i: (\underline{K}^{\MW}_i)^{b_i} \to (\underline{K}^{\MW}_i)^{b_i}$. Such a $\varphi_i$ is determined by an $b_i \times b_i$ square matrix of elements of $\Hom(\underline{K}^{\MW}_i,\underline{K}^{\MW}_i) \cong \GW(k)$ for $i>0$ (Lemma~\ref{HomAbA1_KMWnmn+m}) or of elements of $\ZZ$ for $i=0$. The trace of such a matrix is denoted $\Tr$ and defined to be the sum of the diagonal entries and is viewed as an element of $\GW(k)$ for all $i$.

\begin{pr}\label{pr:TraceDperf}
Let $C_{\ast}$ in $\Ch_{\geq 0}(\Ab_{\Aone}(k))$ be a bounded complex such that $C_n \cong (\underline{K}^{MW}_n)^{b_n}$ for $n>0$ and $C_0 \cong \ZZ^{b_0}$ for some nonnegative integers $b_n$.
Let $\varphi: C_{\ast} \to C_{\ast}$ be a morphism in $D(\Ab_{\Aone}(k))$. The categorical trace of the corresponding map $\varphi$ in \(\DpSWAone\) is computed as
\[\Tr(\varphi) = \sum_{i}  \langle -1 \rangle^i \Tr(\varphi_i).
\]
\end{pr}

\begin{proof}
Let $\epsilon,\eta$ be the evaluation and coevaluation maps as in Proposition~\ref{P:dualinDSWAbAonek}, and $\tau$ the symmetry isomorphism in the category $\DpSWAone$ as in \eqref{eq:tau_DAbAone}. Then $\Tr(\varphi)$ is the composition 
\[
1 \stackrel{\eta}{\longrightarrow} (C_*,N) \otimes \Dual (C_*,N) \stackrel{\varphi \otimes 1}{\longrightarrow} (C_*,N) \otimes \Dual (C_*,N) \stackrel{\tau}{\to}  \Dual (C_*,N) \otimes (C_*,N) \stackrel{\epsilon}{\longrightarrow} 1.
\] This map is represented by a map
\begin{align*}
\underline{K}^{MW}_{t}[t] \to C_\ast   \otimes_{\Ch(\Ab_{\Aone}(k))}  \underline{\Hom}(C_{\ast},\underline{K}^{MW}_{t}[t])\stackrel{\varphi \otimes 1}{\longrightarrow} C_\ast   \otimes_{\Ch(\Ab_{\Aone}(k))}  \underline{\Hom}(C_{\ast},\underline{K}^{MW}_{t}[t]) \\ \to  \underline{\Hom}(C_{\ast},\underline{K}^{MW}_{t}[t]) \otimes_{\Ch(\Ab_{\Aone}(k))} C_\ast  \to \underline{K}^{MW}_{t}[t] 
\end{align*} in $\Ch_{\geq 0}(\Ab_{\Aone}(k))$, where $t$ is as in Proposition~\ref{P:dualinDSWAbAonek} and its proof. This map is concentrated in degree $t$. The degree $t$ sheaves of the complexes $ C_\ast   \otimes_{\Ch(\Ab_{\Aone}(k))}  \underline{\Hom}(C_{\ast},\underline{K}^{MW}_{t}[t])$ and $ \underline{\Hom}(C_{\ast},\underline{K}^{MW}_{t}[t]) \otimes_{\Ch(\Ab_{\Aone}(k))} C_\ast$ are isomorphic to a direct sum over $n$ of $C_n \otimes_{\Aone}  \underline{\Hom}(C_n,\underline{K}^{MW}_{t})$, leading to an expression for  $\Tr(\varphi)$ as the sum over $n$ of maps
\begin{align}
\label{eqn:Trvarphi_n}
\underline{K}^{MW}_{t} \to C_n \otimes_{\Aone}  \underline{\Hom}_{\Ab(k)}(C_n, \underline{K}^{\MW}_t)  \stackrel{\varphi \otimes 1}{\longrightarrow} C_n \otimes_{\Aone}  \underline{\Hom}_{\Ab(k)}(C_n, \underline{K}^{\MW}_t) \\
\stackrel{(-1)^n}{\longrightarrow} \underline{\Hom}_{\Ab(k)}(C_n, \underline{K}^{\MW}_t) \otimes C_n \to \underline{K}^{MW}_{t}.  \nonumber
\end{align}
Tracing through the definitions of Proposition~\ref{P:dualinDSWAbAonek}, the composite \eqref{eqn:Trvarphi_n} is \[
(-\lra{-1})^{tn}(-\lra{-1})^{(t-n)n}(-1)^n \Tr(\varphi_n)=\lra{-1}^n \Tr(\varphi_n).
\]
Thus $\Tr(\varphi) = \sum_{n}  \langle -1 \rangle^n \Tr(\varphi_n)$ as claimed.
\end{proof}

\subsection{Cellular Grothendieck--Lefschetz trace formula}\label{subsection:LTFcellular}

Hoyois \cite{hoyois2015quadratic} proves a quadratic refinement of the Grothendieck--Lefschetz trace formula in the setting of stable motivic homotopy theory, in the sense of relating the trace of an endomorphism $\varphi: X \to X$ of a smooth proper scheme to the fixed points of $\varphi$. The machinery of Morel--Sawant and the above give an expression for the trace in terms of traces of matrices of elements of $\GW(k)$ for simple cellular $X$:

\begin{tm}\label{tm:Cellular_LTF}
Let $X$ be a smooth projective scheme over a field $k$ and suppose that $X$ admits a simple cellular structure. Let \(\varphi \colon X \to X\) be an endomorphism, and let $C_*^{\cell}(\varphi): C_*^{\cell}(X) \to C_*^{\cell}(X) $ be any representative of the canonical chain homotopy class. We have the equality
\[
\Tr(\varphi) =  \sum_i \langle -1 \rangle^i \Tr(C_i^{\cell}(\varphi))
\] in $\GW(k)$.
\end{tm}

Note that with our notational conventions, $C_*^{\cell}(X) \simeq \tilde{C}^{\cell}_*(X_+) \simeq \CcellSW_*(X)$.

\begin{proof}
By Lemma~\ref{Xsmproj_is_dualizable_SW}, the scheme $X$ is fully dualizable and $\Tr(\varphi)$ is a well-defined element of $\GW(k)$. By Proposition~\ref{L:changeofcategory}, there is an equality  $\Tr(\varphi) = \CcellSW_* (\Tr(\varphi))$. By Corollary~\ref{cor:trace_commutes_with_Ccell}, $ \CcellSW_* (\Tr(\varphi)) = \Tr(\CcellSW_*(\varphi))$. By Remark~\ref{rmk:finite_cellular_schemes_Ccell_KMWbi}, $\CcellSW_*(X) \simeq C^{\cell}_*(X)$ satisfies the hypotheses of Proposition~\ref{pr:TraceDperf}. Applying Proposition~\ref{pr:TraceDperf} proves the theorem.
\end{proof}

Classical Lefschetz trace formulas in algebraic or topological categories take the form
\[
\sum_i (-1)^i \Tr(\varphi^*\vert_{H^i(X)}) = \Tr(\varphi) = \sum_{x: \varphi(x) = x} \ind_x \varphi
\] under appropriate hypotheses. For simple cellular schemes,  Theorem~\ref{tm:Cellular_LTF} gives a quadratic refinement of the left equality. We will combine this with Hoyois's enrichment of the right equality later in Section \ref{S:Hoyois-trace-formula}.

\section{Rationality of the zeta function} \label{sec:rationality}

In this section, we prove that the logarithmic zeta function of an endomorphism of a scheme with a cellular structure is computed via \(\CcellSW_*(X)\) (Theorem~\ref{thm:rationality_cellular_slick}).  When the cellular structure is moreover simple, we more explicitly show that the logarithmic zeta function is computed by the action on the terms \(C_i^\cell(X)\) of Morel--Sawant's cellular complex (Theorem~\ref{thm:rationality_cellular}).  This, in turn, proves that \(\dlog \zeta^{\Aone}_{X, \varphi}\) is dlog rational (cf.~Definition~\ref{df:dlog_rational}) in this case.

 Let \(X\) be a cellular scheme over a field \(k\) (cf.~Definition~\ref{df:finite_cellular_structure}) with an endomorphism \(\varphi \colon X \to X\). Since \(\CcellSW_*\) is a functor, we obtain an endomorphism
\[ \CcellSW_*(\varphi) \colon \CcellSW_*(X) \to \CcellSW_*(X). \]
Using Definition~\ref{df:log_derivative_char_poly_varphi}, we define
\[ \dlog \PcellSW_{X, \varphi}(t) = \dlog P_{\CcellSW_*(\varphi)}(t) = \sum_{m=1}^\infty - \Tr(\CcellSW_*(\varphi)^m)t^{m-1} \]
to be the logarithmic characteristic polynomial of the endomorphism \(\CcellSW_*(\varphi)\).

\begin{tm}\label{thm:rationality_cellular_slick}
Let \(X\) be a smooth projective cellular scheme over a field $k$ and let \(\varphi \colon X \to X\) be an endomorphism.  Then
\[\dlog \zeta^{\Aone}_{X, \varphi}(t)  =  -\dlog \PcellSW_{X, \varphi}(t).\]
\end{tm}
\begin{proof}
We compute
 \begin{align*}
  \dlog \zeta^{\mathbb{A}^1}_{X,\varphi} (t) &= \sum_{m \geq 1} \Tr(\varphi^m) t^{m-1}  & (\text{by definition})\\
  &= \sum_{m \geq 1} \CcellSW_*(\Tr(\varphi^m)) t^{m-1} & (\text{by proof of Proposition~\ref{L:changeofcategory}}) \\
  &= \sum_{m \geq 1} \Tr(\CcellSW_*(\varphi^m)) t^{m-1} & (\text{by Corollary~\ref{cor:trace_commutes_with_Ccell}}) \\
    &= \sum_{m \geq 1} \Tr(\CcellSW_*(\varphi)^m) t^{m-1} & (\text{by functoriality}) \\
    &= - \dlog \PcellSW_{X, \varphi}(t) & (\text{by definition}). &\qedhere
  \end{align*}
\end{proof}

We can make \(\dlog \PcellSW_{X, \varphi}(t)\) more explicit when the cellular structure is simple.  Assume this is the case. Let
\[0 \subsetneq \Omega_0(X) \subsetneq \cdots \subsetneq \Omega_{s}(X) = X\]
be a simple cellular structure on \(X\).  Write \(X_i \colonequals \Omega_{i}(X) \smallsetminus \Omega_{i-1}(X)\).  
Then
\[\CcellSW_i(X) \colonequals (C^\cell_i(X), 0) \quad \quad \text{ and } \quad \quad C^\cell_i(X)\cong \begin{cases}
(\underline{K}^{\MW}_i)^{b_i}   & i>0 \\
 \ZZ^{b_0} & i=0.
\end{cases} \] 
See Remark~\ref{rmk:finite_cellular_schemes_Ccell_KMWbi}. By Corollary~\ref{f:map_pos_cellular_to_map_cCell} or \cite[Corollary 2.43]{MorelSawant}, there is a canonical chain homotopy class of endomorphisms $C^{\cell}_*(\varphi)$ of $C^\cell_\ast(X)$. Choosing a representative, we obtain an endomorphism
\[C^\cell_i(\varphi) \colon C^\cell_i(X) \to C^\cell_i(X)\]
of each term in the complex \(C^\cell_*(X)\). 
Since \(C^\cell_i(X) \simeq (\underline{K}^{MW}_i)^{b_i}\), the endomorphism \(C^\cell_i(\varphi)\) is determined by a \(b_i \times b_i\) square matrix of entries in \(\Hom(\underline{K}^{MW}_i, \underline{K}^{MW}_i) \cong \GW(k)\), which is a commutative ring.  We therefore have the usual notion of the characteristic polynomial \(P_{C^\cell_i(\varphi)}(t)\) of \(C^\cell_i(\varphi)\) in the polynomial ring over \(\GW(k)\), namely 
\[
P_{C^\cell_i(\varphi)}(t) = \det (1- t C^\cell_i(\varphi)).
\]

\begin{tm}\label{thm:rationality_cellular}
Let \(X\) be a smooth projective simple cellular scheme over a field $k$ and let $\varphi: X \to X$ be an endomorphism. Then
\[\dlog \zeta^{\Aone}_{X, \varphi} = \sum_{i = - \infty}^\infty -\langle -1 \rangle^i \frac{d}{dt}\log P_{C^\cell_i(\varphi)}(t). \]
\end{tm}
\begin{proof}
 We have
 \begin{align*}
  \dlog \zeta^{\mathbb{A}^1}_{X,\varphi} (t)  &= \sum_{m \geq 1} \Tr(\CcellSW_*(\varphi^m)) t^{m-1} & (\text{by Theorem~\ref{thm:rationality_cellular_slick}}) \\
  &= \sum_{m \geq 1}\sum_{i = - \infty}^\infty  \langle -1 \rangle^i \Tr(\CcellSW_i(\varphi^m)) t^{m-1} & (\text{by Proposition~\ref{pr:TraceDperf}}) \\
    &= \sum_{m \geq 1}\sum_{i = - \infty}^\infty \langle -1 \rangle^i \Tr(C^\cell_i(\varphi^m)) t^{m-1} & (\text{by construction}) \\
    &= \sum_{m \geq 1}\sum_{i = - \infty}^\infty \langle -1 \rangle^i \Tr(C^\cell_i(\varphi)^m) t^{m-1} & (\text{by Corollary~\ref{f:map_pos_cellular_to_map_cCell}} )\\
    &= \sum_{i = - \infty}^\infty -\langle -1 \rangle^i \frac{d}{dt}\log P_{C^\cell_i(\varphi)}(t) & (\text{by Lemma~\ref{lem:linear_algebra}}). & \qedhere
 \end{align*}
\end{proof}

\begin{rmk} \label{rmk:dlogrationality}
If a power series $\Phi(t)$ is dlog rational in the sense of Definition \ref{df:dlog_rational}, it is tempting to write $\Phi = \dlog \prod_{j}P_j^{c_j}$, where the power operation should be a suitable power structure on the Grothendieck--Witt ring, compatible with logarithmic derivatives in the sense that for $r\in \GW(k)$ and $f\in 1 + t\GW(k)[[t]]$ we would have
$$\dlog f^{r} = r\dlog f.$$ 

In fact, it is not hard to check that such a power structure does not exist on $\GW(k)$ if $k^*$ has non-square elements. Indeed, if we could define $$ (1 + t)^r := 1 + a_1 t + a_2 t^2\ldots,$$
for every $r\in \GW(k)$, for appropriate coefficients $a_1,a_2,\ldots\in \GW(k)$ (which are functions of $r$), then we quickly see that, together with the log-derivative compatibility condition, this imposes the relations $a_1 = r$ and $2a_2 = r^2 - r$. 

When $r = \langle u\rangle$ for a non-square $u$, we have $r^2 -r = \langle u \rangle^2 - \langle u\rangle = 1 - \langle u\rangle$, so that the condition on $a_2$ implies that $2a_2 = 1-\langle u\rangle$. Since the discriminant of $2a_2$ is $1/u$, the discriminant of $a_2$ would give a square root of $u$ in $k$, which is a contradiction.

Formally inverting $2$ for $k$ a finite field yields $\GW(k)[\frac{1}{2}] \cong \Z[\frac{1}{2}]$, so we would obtain nothing more than the classical zeta function by taking this method. We instead use the logarithmic derivative (but see also Remark~\ref{rmk:dlog_rationality_R_again}). 
\end{rmk}

\section{$\Aone$-logarithmic zeta functions and real points}\label{Section:Rpoints}

 Let $r^{\R}: \SH(\R) \to \SH$ denote the real realization functor from the stable $\Aone$-homotopy category over $\R$ to the topological stable homotopy category \cite[Section 10]{bachmann-real-etale} . As above, we use the isomorphism $\End(1_{\SH(\R)}) \cong \GW(\R)$ given by the $\Aone$-degree of Morel. This $\Aone$-degree
 has the beautiful property (see e.g., \cite[Prop 3.1.3]{AsokFaselWilliams_Suslin-Herewicz}) that it encodes the topological degree of the corresponding map on real points: 
\begin{equation}\label{eq:sgn_deg}
\sgn \deg^{\Aone}(f) = \deg^{\Top} r^{\R}(f),
\end{equation}
where $\deg^{\Top}$ denotes the topological degree of a self map of a sphere, and $\sgn: \GW(\R) \to \Z$ is the signature homomorphism, defined $\sgn(a \lra{1} + b \lra{-1}) = a-b $ for all $a,b \in \Z$. Thus the map
\[
r^{\R}(\End_1) \colon \GW(\R) \cong \End(1_{\SH(\R)}) \rightarrow \End(1_\SH) \cong \Z
\] induced by the real realization on endomorphisms of the unit objects equals $\sgn: \GW(\R) \to \Z$.

\begin{lm}\label{sgnTr=Tr(R)}
Suppose that $\varphi: X \to X$ is an endomorphism of a smooth scheme $X$ over $\R$. Let $X(\R)$ denote the real points of $X$, viewed as a real manifold, and $\varphi(\R): X(\R) \to X(\R)$ the corresponding endomorphism. Then 
\[
\sgn \Tr (\varphi) = \Tr(\varphi(\R)).
\]
\end{lm}

\begin{proof}
The real realization functor $r^{\R}$ is symmetric monoidal by \cite[Section 4]{heller2016galois}. Thus $r^{\R}\Tr (\varphi) = \Tr(r^{\R}(\varphi)) = \Tr(\varphi(\R))$ by, e.g., Proposition~\ref{HrespectDtr}. Then $r^{\R} \Tr(\varphi) = \sgn \Tr(\varphi)$ by \eqref{eq:sgn_deg}.
\end{proof}

\begin{pr}\label{pr:sigdlogA1_R}
Suppose that $\varphi: X \to X$ is an endomorphism of a smooth projective scheme $X$ over $\R$. Then in $\Z[[t]]$ there is an equality
\[
\sgn \dlog \zeta^{\Aone}_{X, \varphi} = \frac{d}{dt} \log \prod_i (P_{\varphi(\R)\vert H^i_{\Top}}(t))^{(-1)^{i+1}}
\] where $P_{\varphi(\R)\vert H_{\Top}^i}(t)$ denotes the characteristic polynomial
\[
P_{\varphi(\R)\vert H^i}(t) = \det(1-t \varphi(\R)\vert H^i_{\Top}(X(\R);\Z))
\] of the action of $\varphi(\R)$ on the singular cohomology of $X(\R)$.
\end{pr}

\begin{proof}
Let $C^*_{\Top}: H \to D(\ZZ)$ denote the singular cochain functor from the homotopy category of topological spaces to the derived category of $\ZZ$-modules, which is symmetric monoidal. Since $D(\ZZ)$ admits a tensor inverse to $C^*_{\Top}(S^1)$, there is an induced symmetric monoidal functor $C^*_{\Top}: \SW \to D(\ZZ)$ from the topological Spanier--Whitehead category (see Proposition~\ref{pr:SWfunctor}).  We compute

 \begin{align*}
  \sgn \dlog \zeta^{\Aone}_{X, \varphi} &=  \sum_{m \geq 1} \Tr(\varphi(\R)^m) t^{m-1} && (\text{by Lemma~\ref{sgnTr=Tr(R)}}) \\
  &= \sum_{m \geq 1} C^*_{\Top}Tr(\varphi(\R)^m) t^{m-1} && ( C^*_{\Top}:\End(1_{\SW}) \to \ZZ \text{ is an isomorphism} )\\
  &= \sum_{m \geq 1}Tr( C^*_{\Top} (\varphi(\R)^m)) t^{m-1} && (C^*_{\Top} \text{ is symmetric monoidal}) \\
  &= \sum_{m \geq 1}Tr( H^*_{\Top} (\varphi(\R)^m)) t^{m-1} \\
   &= \sum_{m \geq 1}\sum_{i = - \infty}^\infty ( -1 )^i \Tr(H^i_{\Top}(\varphi(\R)^m)) t^{m-1} \\
   &= \sum_{i = - \infty}^\infty -( -1 )^i \frac{d}{dt}\log P_{\varphi(\R)\vert H_{\Top}^i}(t) && (\text{by Lemma~\ref{lem:linear_algebra}}). \qedhere
 \end{align*}
\end{proof}

\begin{ex}\label{ex:same_classical_different_enriched_log_zeta}
The enriched logarithmic zeta function can distinguish between schemes even when its rank, which is a non-enriched logarithmic zeta function, cannot. Consider the smooth projective $\R$-schemes $X:= \Res_{\C/\R} \mathbb{P}^1$ and $Y:= \mathbb{P}^1 \times \mathbb{P}^1$, and equip both with their identity endomorphisms. Applying real points, we have $X(\R) \cong \mathbb{P}^1(\C) \cong S^2$ and $Y(\R) \cong\mathbb{P}^1(\R) \times \mathbb{P}^1(\R) \cong S^1 \times S^1 $. By Proposition~\ref{pr:sigdlogA1_R},
 \begin{align*}
\sgn \dlog \zeta^{\Aone}_{X,1}  &= \frac{d}{dt} \log \frac{1}{(1- t)(1- t)}, \qquad \textrm{ and} \\
\sgn \dlog \zeta^{\Aone}_{Y, 1}  &= \frac{d}{dt} \log \frac{1}{(1- t)(1- t)} - \frac{d}{dt} \log \frac{1}{(1- t)^2}.
\end{align*}
In particular, $\dlog \zeta^{\Aone}_{X, 1} \neq \dlog \zeta^{\Aone}_{Y,1}$. By contrast, the non-enriched logarithmic zeta functions are equal: we have $X(\C) \cong Y(\C)$ and the ranks of the traces of the identity morphisms of $X$ and $Y$ are $\chi^{\Top}(X(\C)) = \chi^{\Top}(Y(\C)) = 4$.
\end{ex}

For a $\ZZ[1/d]$ scheme $\calX$ and a point $z: \Spec L \to \Spec \ZZ[1/d]$, let $\calX_{L}$ denote the pullback  $X_{L}: = \calX \otimes L$ of $\calX$ along $z$.
Similarly, given an endomorphism $\varphi: \calX \to \calX$, let $\varphi_{L}: \calX_{L}\to \calX_{L}$ denote the corresponding pullback. 

Note there are pullback map on Grothendieck--Witt groups
\[
z^* \colon \GW(\ZZ[1/d]) \to \GW(\F_q) \qquad  \text{and} \qquad \eta_{\R}^* \colon \GW(\ZZ[1/d]) \to \GW(\R).
\] Moreover, for $d=1$, the map $\eta_{\R}^*$ is the isomorphism $\GW(\ZZ) \stackrel{\cong}{\to} \GW(\R)$, which allows us to map elements of $\GW(\R)$ into $\GW(\ZZ[1/d])$. Let $\ZZ[ \lra{-1}, \lra{p}: p \text{ prime dividing } d]$ denote the polynomial ring in finitely many variables over $\ZZ$, where each of $\lra{-1}$ and $\lra{p}$ are viewed as variables. The notation defines an evident map $\ZZ[ \lra{-1}, \lra{p}: p \text{ prime dividing } d] \to \GW(\ZZ[1/d])$, which is a surjection by \cite[Lemma 5.6]{BW-A1Eulerclasses}. By a slight abuse of notation, we will also apply the maps $z^*$ and $\eta_{\R}^*$ to elements of the ring $\ZZ[ \lra{-1}, \lra{p}: p \text{ prime dividing } d]$.

\begin{tm}\label{pr:lift_to_R}
Let $\calX \to \Spec \ZZ[1/d]$ be smooth and proper for $d=1$ or $d$ even and let $\varphi: \calX \to \calX$ be an endomorphism. Let $z$ be a closed point of $\ZZ[1/d]$ with residue field $\F_p$. Then there is a a power series $\zeta$ in $\ZZ[ \lra{-1}, \lra{p}: p \text{ prime dividing } d][[t]]$ such that 
\[
\eta_{\R}^*  \zeta = \dlog \zeta^{\Aone}_{\calX_{\R}, \varphi_{\R}} \qquad \text{and} \qquad z^* \zeta = \dlog \zeta^{\Aone}_{\calX_{\F_p}, \varphi_{\F_p}}.
\]
\end{tm}

\begin{proof}
The scheme $\calX$ is dualizable in $\SH(\ZZ[1/d])$ by \cite[Theorem 3.4.2]{DDO-stable_homotopy_infty}. For $d$ even, we have a Hermitian K-theory spectrum $\KO \in \SH(\Z[1/d])$ with $[1_{\Z[1/d]}, \KO] \cong \GW(\Z[1/d])$ \cite{hornbostel2005a1}, and $f^* \KO = \KO$ for $f=z^*, \eta_{\R}^*$(and more generally)  \cite[Theorem 1.2]{panin2010motivic}. Let $u: 1_{\Z[1/d]} \to \KO$ denote the unit map of the ring spectrum $\KO$. Then $u \wedge \Tr(\varphi^m) \in [1_{\Z[1/d]}, \KO] \cong \GW(\Z[1/d])$. Choose preimages $\widetilde{(u \wedge \Tr(\varphi^m))}$ in $\ZZ[ \lra{-1}, \lra{p}: p \text{ prime dividing } d]$ of $u \wedge \Tr(\varphi^m)$ under the surjection $$\ZZ[ \lra{-1}, \lra{p}: p \text{ prime dividing } d] \to \GW(\ZZ[1/d]).$$  Define $\zeta$ by
\[
\zeta := \sum_{m=1}^{\infty} \widetilde{(u \wedge \Tr(\varphi^m))} t^{m-1}
\]

Let $f: \Spec L \to \Spec \ZZ[1/d]$ denote either of the maps $z: \Spec \F_p \to \Spec \ZZ[1/d]$ or $\eta_{\R}: \Spec \R \to \Spec \ZZ[1/d]$. The proposition then follows from the following equalities in $\GW(L)$:
\begin{align*}
f^* ( u \wedge \Tr(\varphi^m)) &=  f^* u \wedge  f^*( \Tr(\varphi^m)) & (f^* \text{ is symmetric monoidal}) \\
&= u \wedge \Tr(f^*\varphi^m) & (f^*\KO \simeq \KO \text{ and  Proposition~\ref{HrespectDtr}})\\
&=  \Tr(f^*\varphi^m) & ([1_L, 1_L] \stackrel{\cong}{\to}[1_L, \KO]  ). & 
\end{align*}

For $d=1$, replace $\KO$ with the spectrum $\KO'$ in $\SH(\ZZ)$ constructed in \cite[\S 3.8.3]{BHeta_periodic}. By \cite[Lemma 3.38(2)]{BHeta_periodic}, $[1_{\ZZ}, \KO'] \cong \GW(\ZZ)$. Let $u: 1 \to \KO'$ denote the unit inherited from the orientation of \cite[Lemma 3.38(1)]{BHeta_periodic}, and define $\zeta$ as above. By \cite[Lemma 3.38(3)]{BHeta_periodic},  $\eta_{\R}^* \KO' = \KO$ and for $p$ odd, $z^* \KO' = \KO$. It follows as before that $\eta_{\R}^*  \zeta = \dlog \zeta^{\Aone}_{\calX_{\R}, \varphi_{\R}}$ and $z^* \zeta = \dlog \zeta^{\Aone}_{\calX_{\F_p}, \varphi_{\F_p}}$ for $p$ odd. 

For $p=2$, $\GW(\F_p) \cong \ZZ$ via the rank. Thus the claim that $z^* \zeta = \dlog \zeta^{\Aone}_{\calX_{\F_p}, \varphi_{\F_p}}$ is equivalent to showing that $\rk \Tr(\eta_{\R}^* \varphi^m) = \rk \Tr(z^* \varphi^m)$. This follows from a similar argument to the above with the K-theory spectrum replacing $\KO$. In more detail, let $u: 1_{\ZZ} \to \K$ denote the unit for the K-theory spectrum. $f^* \K = \K$ by the geometric model of K-theory using Grassmannians \cite{morelvoev}. Thus $f^* ( u \wedge \Tr(\varphi^m)) = f^* u \wedge  f^*( \Tr(\varphi^m)) = u \wedge \Tr(f^*\varphi^m)$. We have $ u \wedge \Tr(f^*\varphi^m) = \rk \Tr(f^*\varphi^m)$ via the identification of the unit $[1_L, 1_L] \to [1_L, \K]$ with the rank map. The map $f^*:\ZZ \stackrel{\rk}{\cong} [1_\ZZ, \K] \to [1_L , \K] \stackrel{\rk}{\cong} \ZZ$ (via the rank isomorphisms on $\K_0$ for $\ZZ$ and $L$) is the identity on $\ZZ$. Thus $\rk \Tr(\eta_{\R}^* \varphi^m) = \rk \Tr(z^* \varphi^m)$, proving the proposition.

\end{proof}

\begin{rmk}\label{rmk:creditthenine}
The proof of Theorem~\ref{pr:lift_to_R} depends in a crucial way on \cite{CDHHLMNNS-3} though the use of the results of \cite{BHeta_periodic}.
\end{rmk}

\begin{co}\label{co:lift_R_computation}
For $\varphi:\calX \to \calX$ an endomorphism of a smooth and proper $\ZZ$-scheme.
\[
\dlog^{\Aone} \zeta_{\calX_{\F_q}, \varphi_{\F_q}} = z^* \dlog^{\Aone} \zeta_{\calX_{\R}, \varphi_{\R}}.
\]
\end{co}

\begin{proof}
Follows from Theorem~\ref{pr:lift_to_R} and the isomorphism $\eta_{\R}^*: \GW(\ZZ) \stackrel{\cong}{\to} \GW(\R)$.
\end{proof}

\begin{proof}(of Theorem~\ref{thmintro:dlogZA1-lift-of-Frobenius-smooth-proper-Z})
The lift of $ \dlog \zeta^{\Aone}_{X,\F_q}$ to $\GW(\ZZ[\frac{1}{d}])[[t]]$ is $\sum_m (u \wedge \Tr(\varphi^m)) t^{m-1}$ as in the proof of Theorem~\ref{pr:lift_to_R}. Theorem~\ref{thmintro:dlogZA1-lift-of-Frobenius-smooth-proper-Z} then follows from Theorem~\ref{pr:lift_to_R} and  Proposition~\ref{pr:sigdlogA1_R}.

\end{proof}

\begin{ex}\label{ex:toric_varieties}
If $\varphi: X \to X$ is the relative Frobenius of a smooth, proper toric variety $X$ over $\F_p$, the Frobenius lifts to $\varphi:\calX \to \calX$ with $\calX$ smooth and proper over $\ZZ$ (see, e.g., \cite[Section 3.4]{Buch_Thomsen-Frobenius_toric_variety} or \cite[Section 2.4]{Borger-field_one_element}). (The criteria for smoothness and properness of toric varieties shows that the lift $\calX$ given in the references is indeed smooth and proper.) Theorem~\ref{thmintro:dlogZA1-lift-of-Frobenius-smooth-proper-Z} computes the logarithmic $\Aone$-zeta function of these varieties:
\begin{align*}
\dlog \zeta^{\Aone}_{X,\F_p} =& \sum_i \frac{(-1)^{i+1}}{2} (\langle 1 \rangle - \langle -1 \rangle) \frac{d}{dt}\log \det(1-t \varphi(\R)\vert H^i_{\Top}(\calX(\R);\Z))
\\ + &\sum_i \frac{(-1)^{i+1}}{2} (\langle 1 \rangle + \langle -1 \rangle) \frac{d}{dt}\log \det(1-t \varphi(\C)\vert H^i_{\Top}(\calX(\C);\Z)),
\end{align*}
\end{ex}

\begin{rmk}\label{rmk:dlog_rationality_R_again}
Since $\GW(\Z) \cong \GW(\R)$ is torsion-free, one can construct a power structure on $\GW(\Z)\otimes \Q$ as in Remark~\ref{rmk:dlogrationality} giving an $\Aone$-zeta function $\zeta^{\Aone}_{\calX/\Z,\varphi}$ with the appropriate logarithmic derivative. By Proposition~\ref{pr:sigdlogA1_R}, 
\begin{equation}\label{eq:degZeta=sum_Betti_number_X(R)}
\deg \sgn \zeta^{\Aone}_{\calX/\Z,\varphi} = \sum_i (-1)^{i+1} b_i
\end{equation} where $b_i$ is the $i$th Betti number of $X(\R)$. Equation (\ref{eq:degZeta=sum_Betti_number_X(R)}) produces non-trivial lower bounds on the Betti numbers of $\calX(\R)$. In the situation of Corollary~\ref{co:lift_R_computation}, this $\Aone$-zeta function also determines $\dlog^{\Aone} \zeta_{\calX_{\F_q}, \varphi_{\F_q}}$.
\end{rmk}

\begin{ex}\label{ex:bounds:b_i}
In the situation of Theorem~\ref{thmintro:dlogZA1-lift-of-Frobenius-smooth-proper-Z}, $\calX(\R) = \emptyset$ implies that $\sgn \dlog \zeta^{\Aone}_{X,\F_p} = 0$. As always, $\rk \dlog \zeta^{\Aone}_{X,\F_p} = \frac{d}{dt} \log \zeta_{X,\F_p}$. For $d=1$, it follows that $\dlog \zeta^{\Aone}_{X,\F_p} = \sum_m \frac{\vert X(\F_{p^m}) \vert}{2} (\langle 1 \rangle + \langle -1 \rangle) t^{m-1}$. On the other hand, we must have $\dlog \zeta^{\Aone}_{X,\F_p}  =  \sum_{m \geq 1}\left( \sum_{i \vert m} \alpha(i) \Tr_{\F_{q^i/\F_q}}\lra{1} \right)t^{m-1}$ in $\GW(\F_p)[[t]]$ by Hoyois's fixed point theorem \eqref{AoneZeta_from_points}. Suppose $p \equiv 3 \mod 4$. This implies that $\frac{\vert X(\F_{p^m}) \vert}{2}$ is an integer and $\sum_{\substack{i|m\\ i\ \text{even}}}\frac{1}{i}\sum_{d|i} \mu(d)|X(\F_{q^{i/d}})| \equiv \frac{\vert X(\F_{p^m}) \vert}{2} \mod 2$. This condition is non-trivial and necessary for $\calX(\R) = \emptyset$. It is not satisfied by $\bbb{P}^2$. 

\end{ex}

When the Frobenius does not lift, we wish to use Theorems~\ref{thm:rationality_cellular_slick} and \ref{thm:rationality_cellular} to relate the logarithmic $\Aone$-zeta function to the real points of a lift over a ring with a real place. 

Let $\underline{I}^n$ denote the unramified sheaf corresponding to the $n$th power of the fundamental ideal $I = \ker(\rk: \GW(-) \to \Z)$ \cite[Example 3.34]{A1-alg-top}. 
Recall that the singular cochain complex of a topological space $X$ is denoted $C_{\Top}^*(X;\Z)$. For $C_*,C_*'$ in $\Ch(\Ab_{\Aone}(k))$, we may view $\Hom_{\Ch(\Ab_{\Aone}(k))}(C_*, C_*')$ as an element of $\Ch(\Z)$.

\begin{pr}\label{pr:rRcellHomXsimpCell}

Let $X$ be a cellular smooth scheme over $\R$ of dimension $d$ and let $n>d$ be an integer. There is a quasi-isomorphism $$\Hom_{\Ch(\Ab_{\Aone}(k))}(C_*^\cell(X), \underline{I}^n) \simeq C_{\Top}^*(X(\R);\Z).$$

\end{pr}

\begin{proof}
Both $C^* :=\Hom_{\Ch(\Ab_{\Aone}(k))}(C_*^\cell(X), \underline{I}^n)$ and $C_{\Top}^*(\calX(\R);\Z)$ are bounded below cohomological complexes of $\Z$-modules. Since $\Z$ is a hereditary ring, it is sufficient to show that $C^*$ and $ C_{\Top}^*(\calX(\R);\Z)$ have isomorphic homology groups. By \cite[Proposition 2.27]{MorelSawant}, 
\[
H^*(\Hom_{\Ch(\Ab_{\Aone}(k))}(C_*^\cell(X), \underline{I}^n)) \cong H^*_{\nis}(X;\underline{I}^n).
\] By \cite[Remark 2.28]{MorelSawant} and \cite{A1-alg-top}, for any smooth $k$-scheme $Y$ and strictly $\mathbb{A}^1$-invariant sheaf $\mathcal{F}$ on $\Sm_k$, the canonical morphism $H^*_{\Zar}(Y;\calF) \cong H^*_{\nis}(Y;\calF) $ is an isomorphism for $* \geq 0$. Thus $H^*_{\nis}(X;\underline{I}^n) \cong H^*_{\Zar}(X;\underline{I}^n)$. A theorem of J.~Jacobson \cite[Corollary 8.11]{Jacobson-RCoh_I} gives that the signature induces an isomorphism $H^*_{\Zar}(X;\underline{I}^n) \cong H^*_{\Top}(X(\R);\Z)$ for $n > \dim X$.
\end{proof}

Let $X$ be a simple cellular scheme over $\R$. By Remark~\ref{rmk:finite_cellular_schemes_Ccell_KMWbi}, the complex $C^{\cell}_*(X)$ is quasi-isomorphic to a complex of the form
\begin{equation}\label{CcellstrictCellX}
\cdots \leftarrow 0 \leftarrow \Z^{b_0} \leftarrow (\underline{K}^{\MW}_1)^{b_1} \leftarrow (\underline{K}^{\MW}_2)^{b_2 } \leftarrow \cdots \leftarrow (\underline{K}^{\MW}_s)^{b_s} \leftarrow 0 \leftarrow \cdots.
\end{equation} 

\begin{pr}\label{pr:rRcellHomXstrictCell}
Let $X$ be a simple cellular scheme over $\R$ and let $b_i$ be as in \eqref{CcellstrictCellX} with $b_i=0$ for $i>s$. Then there is a complex $C_*$ in $\Ch_{\geq 0}(\Z)$ with $C_i \cong \Z^{b_i}$ and a quasi-isomorphism 
\[
C_* \simeq  C^{\Top}_*(X(\R);\Z)
\] between $C_*$ and the singular chains on the real manifold $X(\R)$.
\end{pr}

\begin{proof}
We claim that for $n>s$, we may take $C_*$ to be $\Hom_{\Ch(\Z)}(\Hom_{\Ch(\Ab_{\Aone}(k))}(C_*^\cell(X), \underline{I}^n),\Z)$, so the dual of $\Hom_{\Ch(\Ab_{\Aone}(k))}(C_*^\cell(X), \underline{I}^n)$ has the claimed properties. Let $(-)_{-1}$ denote the $-1$ construction of Voevodsky \cite[p 33]{A1-alg-top} and let $(-)_{-i}$ denote the result of applying the $-1$ construction $i$ times. For $n>s\geq i \geq 0$, the internal $\Hom$ in $\Ab_{\Aone}(k)$ from $\underline{K}^{\MW}_i$ to $\underline{I}^n$ is computed  via $\Hom_{\Ab_{\Aone}(k)}(\underline{K}^{\MW}_i, \underline{I}^n) \cong  \underline{I}_{-i}^n \cong \underline{I}^{n-i}$ by \cite[Lemma 5.1.3]{AsokWickelgrenWilliams_SEHP} and \cite[Proposition 2.9]{Asok_Fasel_alg_vector_bundles}. Since we have $ \GW(\R) \cong \Z[\langle -1 \rangle]/(\langle -1 \rangle^2 - 1)$, the ideal $I^n \subset \GW(\R)$ is the principal ideal generated by $(\langle -1 \rangle - 1)^n$. Thus we have $ \underline{I}^{n-i}(\R) \cong \Z$, so $C_i \cong \Z^{b_i}$. By Proposition~\ref{pr:rRcellHomXsimpCell}, we have $C_* \simeq  C^{\Top}_*(X(\R);\Z)$.
\end{proof}

Note that the definition of a strict cellular structure (see Definition \ref{df:strict_cellular_structure}) makes sense over an arbitrary base scheme. Given a smooth proper strictly cellular scheme $\calX \to \Spec A$ over a ring $A$ and a real point $\eta: \Spec \R \to A$, define $\calX_{\eta}(\R)$ to be the real manifold associated to the real points of $\calX \times_A \R$.

\begin{pr}\label{pr:RandFqPlace}
Suppose $\calX \to \Spec A$ is a smooth projective strictly cellular scheme with points $z: \Spec \F_q \to A$ and $\eta: \Spec \R \to A$. Let $X:=\calX \times_A \F_q$ denote the pullback of $\calX$ along $z$. Let $\varphi: X \to X$ denote an endomorphism of $X$. Then
\begin{enumerate}
\item\label{pr:RandFqPlace:itemR} There is a complex $C_*$ in $\Ch_{\geq 0}(\Z)$ quasi-isomorphic to the topological chains on $\calX_{\eta}(\R)$
\[
C_* \simeq C^{\Top}_*(\calX_{\eta}(\R);\Z).
 \]  of the form
 \[
 \cdots \leftarrow 0 \leftarrow \Z^{b_0} \leftarrow \Z^{b_1} \leftarrow \Z^{b_2 } \leftarrow \cdots \leftarrow \Z^{b_s} \leftarrow 0 \leftarrow \cdots
 \] where $b_i$ is the number of $i$-cells of $\calX$.

 \item \label{pr:RandFqPlace:itemRandFq} The $\Aone$-logarithmic zeta function of $\varphi$ is given by the formula \[\dlog \zeta^{\Aone}_{X, \varphi} = \sum_{i = - \infty}^\infty -\langle -1 \rangle^i \frac{d}{dt}\log P_{C^\cell_i(\varphi)}(t) \] where \[
P_{C^\cell_i(\varphi)}(t) = \det (1- t C^\cell_i(\varphi)).
\] and $C^\cell_i(\varphi)$ is a square matrix of elements of $\GW(\F_q)$ of size $b_i \times b_i$.
 \end{enumerate}

\end{pr}
\begin{proof}
Let $C_*$ denote the complex of Proposition~\ref{pr:rRcellHomXstrictCell} applied to the strictly cellular scheme $\calX \times_{\eta} \R$ over $\R$. The condition of \eqref{pr:RandFqPlace:itemR} is satisfied by Proposition~\ref{pr:rRcellHomXstrictCell}. Moreover, $b_i$ is the number of connected components of $\Omega_i(X)\smallsetminus \Omega_{i-1}(X)$ in the decomposition \[ \Omega_i(X) \smallsetminus \Omega_{i-1}(X) \cong \coprod_{\alpha \in \beta_i} \bA^{n-i}.\] The formula  \eqref{pr:RandFqPlace:itemRandFq} then follows from Theorem~\ref{thm:rationality_cellular}. 
\end{proof}

\section{Computing $\Aone$-logarithmic zeta functions and examples} \label{sec:computingexamples}

In this section, we describe two methods to explicitly compute examples of $\Aone$-zeta functions. First, for schemes with a strict cellular structure over a finite field, we explicitly compute the enriched logarithmic zeta function of Frobenius endomorphisms using Theorem \ref{thm:rationality_cellular}; this gives examples of the enriched logarithmic zeta function for varieties like projective space and Grassmannians.

Then over a finite field, we also use Hoyois' enriched Grothendieck--Lefschetz trace formula and M\"obius inversion to obtain the coefficients of the enriched logarithmic zeta function of the Frobenius from those of the logarithmic derivative of the classical zeta function. We may use this method to compute the enriched logarithmic zeta function for non-cellular schemes, such as elliptic curves.

\subsection{\(\Aone\)-zeta function of Frobenius endomorphisms using Theorem \ref{thm:rationality_cellular}}
 Let $k$ be a finite field $\mathbb{F}_q$ and $X$ be a smooth proper $k$-scheme with a strict cellular structure
\begin{equation}\label{strict_cellular_structure_X_equn}
\emptyset  = \Sigma_{-1}(X) \subset \Sigma_0(X) \subset \Sigma_1(X) \subset \cdots \subset \Sigma_{n-1}(X) \subset \Sigma_n(X) = X
\end{equation}
defined over \(k\).
As above, we set $\Omega_i(X) := X \setminus \Sigma_{n-i-1}(X)$. 

Following \cite[Lemma 3.14]{A1-alg-top}, for a positive integer $n$, let $n_\epsilon \in \GW(k)$ denote $$n_\epsilon : = \langle 1 \rangle + \langle -1 \rangle + \langle 1 \rangle + \cdots + \langle -1 \rangle, $$ the class of the rank $n$ diagonal form with alternating $1$'s and $-1$'s along the diagonal. Note that $n_\epsilon m_\epsilon = (nm)_\epsilon$.

\begin{pr}\label{pr:deg_Ccell_Frob}
Let $X$ be a smooth proper scheme over \(\mathbb{F}_q\) with a strict cellular structure defined over \(\mathbb{F}_q\), and let $\varphi$ be the relative Frobenius. Then $C_{\ast}^{\cell}(\varphi)$ is the map of complexes which in degree $i$ $$C_{i}^{\cell}(\varphi): C_{i}^{\cell}(X) \to C_{i}^{\cell}(X)$$
is multiplication by $q^i_\epsilon$.
\end{pr}

\begin{proof}
We use the above notation for the strict cellular structure \eqref{strict_cellular_structure_X_equn} on $X$. Since the relative Frobenius $\varphi$ gives an endomorphism of the smooth pair $(\Omega_i(X), \Omega_i(X)\setminus X_i)$, it determines a map  $\varphi: \Th_{X_n}( \nu_n) \to \Th_{X_n}( \nu_n)$ by purity \cite[Theorem 2.23]{morelvoev}. The map of complexes $C_{\ast}^{\cell}(\varphi)$ in degree $i$ is given by $\widetilde{H}^{\Aone}_{i}(\varphi):  \widetilde{H}^{\Aone}_{i} (\Th_{X_i}( \nu_i)) \to \widetilde{H}^{\Aone}_{i} (\Th_{X_i}( \nu_i))$. 

 For each $m \in M_i$, choose an $\mathbb{F}_q$-point $p_m$  in the corresponding connected component of $X_i \cong \coprod_{m \in M_i} X_{im}$ and $X_{im} \cong \mathbb{A}^{n-i}$. The inclusion 
\[
\Th_{p_m} \nu_i\hookrightarrow \Th_{X_{im}} \nu_i 
\] is an $\Aone$-weak equivalence. 

Since $X$ has a strict cellular structure, the Krull dimension of $X$ is $n$. By \cite[Th\'eor\`eme II.4.10]{sga1}, we may choose a Zariski open subset $U \subset X$ containing $p_m$ and an \'etale map $\psi: U \to \mathbb{A}^n$ such that
\[
\xymatrix{X_{im} \cap U \ar[r]\ar[d] & \ar[d]_{\psi} \Omega_i(X)\cap U\\
 \mathbb{A}^{n-i} \ar[r] & \mathbb{A}^{n}}
\] is a pullback, and the bottom horizontal map is the natural map $\Spec k[y_{i+1}, \ldots, y_n] \to \Spec k[y_{1}, \ldots, y_n] $ given by $(y_{i+1}, \ldots, y_n) \mapsto (0 ,\ldots, 0, y_{i+1}, \ldots, y_n)$. We obtain maps of Thom spaces
\[
\Th_{p_m} \nu_i \to \Th_{X_{im} \cap U} \nu_i \to \Th_{\mathbb{A}^{n-i}} N_{\mathbb{A}^{n-i}} {\mathbb{A}^{n}} 
\] whose composition is an $\Aone$-weak equivalence.

Purity \cite[Theorem 2.23]{morelvoev} defines a canonical $\Aone$-weak equivalence 
\[
 \Th_{\mathbb{A}^{n-i}} N_{\mathbb{A}^{n-i}} \mathbb{A}^{n} \simeq \bA^n/ \bA^n-\bA^{n-i}.
 \]

The map $\Spec k[y_{1}, \ldots, y_i] \to \Spec k[y_{1}, \ldots, y_n] $ given by $(y_{1}, \ldots, y_i) \mapsto ( y_{1}, \ldots, y_i,0 ,\ldots, 0)$ determines an $\Aone$-weak equivalence 
\[
\bA^i/ \bA^i-\bA^{0} \simeq  \bA^n/ \bA^n-\bA^{n-i}
\]
and excision determines an $\Aone$-weak equivalence 
\[
\bA^i/ \bA^i-\bA^0 \simeq \PP^i/\PP^{i-1}.
\]

Composing in the homotopy category, we obtain a canonical (zig-zag) $\Aone$-weak equivalence 
\[
\Th_{X_i}( \nu_i) \simeq \PP^i/\PP^{i-1}
\]
where the relative Frobenius acts compatibly on both sides. The claim is thus equivalent to showing that the map
\[
 \widetilde{H}^{\Aone}_{i}(\varphi): \widetilde{H}^{\Aone}_{i} (\PP^i/\PP^{i-1}) \to \widetilde{H}^{\Aone}_{i} ( \PP^i/\PP^{i-1})
\] induced by $\varphi$ is multiplication by $q_\epsilon^i$. By Morel's Hurewicz theorem, this map $\widetilde{H}^{\Aone}_{i}(\varphi)$ is multiplication by $\deg^{\Aone} \varphi$. The relative Frobenius is the map $$\PP^i/\PP^{i-1} \to \PP^i/\PP^{i-1}$$ $$y_j\mapsto y_j^q$$ for $j=0,\ldots, i$, which is
homotopy equivalent to the $i$-fold smash product of $\PP^1 \to \PP^1$ mapping $y_j \mapsto y_j^q$. The degree of this map is $q_\epsilon^i$. (To see this, first note that the smash product multiplies degrees \cite{morel2004motivic-pi0} so it suffices to prove the claim for $i=1$. The computation for $i=1$ follows in a straightforward manner from computing the B\'ezoutian \cite[Definition 3.4]{Cazanave} which computes the $\Aone$-degree \cite[Theorem 3.6 and Theorem 1.2]{Cazanave}.) The degree and the claim follows.
\end{proof}

Using Proposition \ref{pr:deg_Ccell_Frob} we can give an explicit description of the enriched logarithmic zeta function of a smooth projective scheme over \(\mathbb{F}_q\) equipped with an \(\mathbb{F}_q\)-rational strict cellular structure.  Recall that in this case, for each $i$,
\[C_i^\cell(X) \simeq (\underline{K}^{MW}_i)^{b_i}\]
 for a nonnegative integer \(b_i \geq 0\).  We call the integer \(b_i\) the \defi{rank of \(C_i^\cell(X)\)}. 

\begin{co}\label{co:zeta_strict_cellular}
Let \(X\) be a smooth projective scheme of dimension \(n\) over \(\mathbb{F}_q\) equipped with a strict cellular structure defined over \(\mathbb{F}_q\).  Let \(b_i\) be the rank of \(C_i^\cell(X)\).  Then 
\[ \dlog \zeta^{\Aone}_{X}(t) = \sum_{i=0}^n - \langle -1 \rangle^i \frac{d}{dt} \log (1 - q^i_\epsilon t)^{b_i}. \]
\end{co}
\begin{proof}
Let \(\varphi\) denote the relative Frobenius endomorphism.
By Theorem \ref{thm:rationality_cellular}, it suffices to show that \(P_{C^\cell_i(\varphi)}(t) = (1- q^i_\epsilon t)^{b_i}\).  This follows from the fact that \(C_i^\cell(\varphi)\) is a \(b_i \times b_i\) square matrix, which by Proposition \ref{pr:deg_Ccell_Frob} is multiplication by \(q^i_\epsilon\).
\end{proof}

\begin{rmk}
We think of Corollary \ref{co:zeta_strict_cellular} as morally saying
\[``\left.\zeta^\Aone_X(t) = \frac{1}{\prod_{i \text{ odd}} (1-q^i_\epsilon t)^{b_i \langle -1\rangle} \prod_{i \text{ even}}(1-q^i_\epsilon t)^{b_i}} \right.''.\]
As explained in Remark \ref{rmk:dlogrationality}, this does not make literal sense because there is not a \(\lambda\)-ring structure on \(\GW(\mathbb{F}_q)(t)\) compatible with the logarithmic derivative.
\end{rmk}

Corollary \ref{co:zeta_strict_cellular} yields many examples of enriched logarithmic zeta functions for varieties of natural geometric interest.

\begin{ex}[Enriched logarithmic zeta function of \(\mathbb{P}^n\)]\label{ex:projspacezeta}
Projective space \(\mathbb{P}^n\) has a strict cellular structure with a single copy of \(\mathbb{A}^i\) for each \( 0 \leq i \leq n\).
Applying Corollary \ref{co:zeta_strict_cellular} shows that the logarithmic zeta function of $\mathbb{P}^n$ is given by
 \[ \dlog \zeta^{\mathbb{A}^1}_{\mathbb{P}^n}(t) = \frac{d}{dt} \log \frac{1}{\displaystyle\prod_{0 \leq i \leq n \atop \text{i even}} (1- q_\epsilon^it)} + \langle -1 \rangle\frac{d}{dt} \log \frac{1}{\displaystyle\prod_{1 \leq i \leq n \atop \text{i odd}} (1- q_\epsilon^it)}. \]
\end{ex}

\begin{rmk}[Functional equation for projective spaces]
We record here that the $\mathbb{A}^1$-logarithmic zeta function of projective spaces of odd dimension satisfies a functional equation.

When $n$ is odd, we have
\[ \chi_c^{\mathbb{A}^1}(\mathbb{P}^n) = \frac{n+1}{2} \left( \langle 1 \rangle + \langle -1 \rangle \right) \]
(see \cite[Example~1.6]{Levine-EC}). By Example~\ref{ex:projspacezeta}, we have
\[ \frac{d}{dt} \log \zeta^{\mathbb{A}^1}_{\mathbb{P}^n}(t) = \frac{d}{dt} \log \frac{1}{\prod_{0 \leq i \leq n \atop \text{i even}} (1- \tilde{q}^it)} + \langle -1 \rangle\frac{d}{dt} \log \frac{1}{\prod_{1 \leq i \leq n \atop \text{i odd}} (1- \tilde{q}^it)}. \]
For any $i$ such that $0 \leq i \leq n$, we have
\[ \frac{1}{1- \tilde{q}^i \frac{1}{\tilde{q}^{n}t}} = \frac{\tilde{q}^nt}{\tilde{q}^i \left( \tilde{q}^{n-i}t-1 \right)} = \frac{-\tilde{q}^{n-i}t}{1-\tilde{q}^{n-i}t}.  \]
Combining the last three equations, we see that
\[ \dlog \zeta^{\mathbb{A}^1}_{\mathbb{P}^n}(t) = - \chi_c^{\mathbb{A}^1}(\mathbb{P}^n)t^{-1}+\langle -1 \rangle \dlog \zeta^{\mathbb{A}^1}_{\mathbb{P}^n}\left(\frac{1}{\tilde{q}^{n}t}\right). \]

The functional equation for the regular zeta function is a consequence of Poincar\'e duality for the $\ell$-adic \'{e}tale cohomology groups. Morel and Sawant have interesting conjectures on Poincare duality for their cohomology theories for general smooth projective varieties $X$. It would be interesting to know if an analogous functional equation of the logarithmic zeta function can be deduced from such a statement for general smooth projective varieties. 
\end{rmk}

\begin{ex}[Enriched logarithmic zeta function of \(\mathbb{G}(1,3)\)]
The Grassmannian \(\mathbb{G}(1,3)\) of lines in \(\mathbb{P}^3\) has a strict cellular structure composed of the Schubert cycles, as we now recall.  Fix a full flag \(V_0 \subset V_1 \subset V_2 \subset V_3  = \mathbb{P}^3\) of linear subspaces, where \(V_i \) has dimension \(i\). For \(0 \leq b \leq a \leq 2\), let \(\Sigma_{a,b}\) denote the locus of lines in \(\mathbb{G}(1,3)\) that meet \(V_{2-a}\) in a point and \(V_{3-b}\) in a line.  Let \(\Sigma_{a,b}^\circ\) denote the complement of all other Schubert cycles in \(\Sigma_{a,b}\).  Then, as in \cite[Section 3.3.1]{EisenbudHarris}, we have \(\Sigma_{a,b} \simeq \mathbb{A}^{4-a-b}\), and the filtration by closed subsets
\[\emptyset \subset \Sigma_{2,2} \subset \Sigma_{2, 1} \subset \left(\Sigma_{1,1} \cup \Sigma_{2,0}\right) \subset \Sigma_{1,0} \subset \Sigma_{0,0} = \mathbb{G}(1,3)\]
gives a strict cellular structure as in Definition \ref{df:strict_cellular_structure}\eqref{cond:F}.  Hence we have \(b_2 = 2\) and all other \(b_i=1\) for \(i=0,1, 3, 4\).  Thus
 \[ \dlog \zeta^{\mathbb{A}^1}_{\mathbb{G}(1,3)}(t) = \frac{d}{dt} \log \frac{1}{(1- t)(1-q_\epsilon^2t)^2(1-q_\epsilon^4t)} + \langle -1 \rangle\frac{d}{dt} \log \frac{1}{(1- q_\epsilon t)(1- q_\epsilon^3 t)}. \]
\end{ex}

\begin{ex}[Enriched logarithmic zeta function of $\mathbb{P}^1 \times \mathbb{P}^1$]\label{E:productofproj} 
 The product \(\mathbb{P}^1 \times \mathbb{P}^1\) has a strict cellular structure with a single copy of \(\mathbb{A}^0\), two copies of \( \mathbb{A}^1\) and one copy of \( \mathbb{A}^2 \).
Applying Corollary \ref{co:zeta_strict_cellular} shows that the logarithmic zeta function of $\mathbb{P}^1 \times \mathbb{P}^1$ is given by
 \[ \dlog \zeta^{\mathbb{A}^1}_{\mathbb{P}^1 \times \mathbb{P}^1}(t) = \frac{d}{dt} \log \frac{1}{(1- t)(1- q_\epsilon^2t)} + \langle -1 \rangle\frac{d}{dt} \log \frac{1}{(1- q_\epsilon t)^2}. \]
\end{ex}

\subsection{$\Aone$-logarithmic zeta functions via Hoyois's trace formula}\label{S:Hoyois-trace-formula}

In \cite[Theorem 1.3]{hoyois2015quadratic}, Hoyois provides a quadratic refinement of the Grothendieck--Lefschetz trace formula in the setting of stable motivic homotopy theory, which gives us a procedure to compute the coefficients of the enriched logarithmic zeta function $\dlog \zeta^{\Aone}_{X,\varphi}$ from those of $\dlog \zeta_X(t)$.

We introduce some notation first. For an endomorphism $\varphi:X\to X$ of a scheme $X$, we denote by $X^{\varphi}$ its scheme of fixed points. Given a finite separable field extension $L/K$, the classical trace map
$\Tr_{L/K}:L\to K$ induces a \textit{Transfer}
\begin{equation}\label{TrGWeq}
\Tr_{L/K}:\GW(L)\to \GW(K)
\end{equation}
by sending a bilinear form $b:V\times V \to L$ to the form $Tr_{L/K}\circ b:V\times V\to K$ of rank $[L:K] \rk (b)$  (see \cite[The cohomological transfer, Chapter 4]{A1-alg-top} \cite[Lemma 2.3]{calmes2014finite}).

Hoyois' main theorem then has the following consequence:

\begin{pr}\cite[Corollary 1.10]{hoyois2015quadratic} Let $k$ be a field, let $X$ be a smooth proper $k$-scheme, and let $\varphi:X\to X$ be a $k$-morphism with \'etale fixed points. Then
$$\Tr(\varphi) = \sum_{x\in X^{\varphi}} \Tr_{\kappa(x)/k}\langle \det(\id - d\varphi_x)\rangle.$$
\end{pr}

This result gives a computation of the $\Aone$-logarithmic zeta function of a smooth proper scheme over a finite field in terms of its point counts.

\begin{tm}\label{T:enrichedzetaviaHoyois} 
Let $X$ be a smooth proper scheme over $\F_q$ and let $\varphi: X\to X$ be the relative Frobenius morphism.  Let $u$ denote a non-square in $\F_{q}$. Then
the $\Aone$-logarithmic zeta function $\dlog \zeta^{\Aone}_{X, \varphi} $ is computed by the following formula:
\begin{align*}
\dlog \zeta^{\Aone}_{X, \varphi} = \sum_m  &\left( \left(\sum_{ \substack{i|m\\ i \text{ even}}} \left(\frac{1}{i}\sum_{d|i} \mu(d)|X(\F_{q^{i/d}})| \right) (i-1)\langle 1 \rangle + \langle u \rangle) \right)\right. \\
& \ \ \   \left. + \left( \sum_{\substack{i|m\\ i \text{ odd}}} \left(\frac{1}{i}\sum_{d|i} \mu(d)|X(\F_{q^{i/d}})| \right) i\langle 1 \rangle \right) \right) t^{m-1},
\end{align*} and $\frac{1}{i}\sum_{d|i} \mu(d)|X(\F_{q^{i/d}})| $ is an integer.
\end{tm}

\begin{proof}
For $X$ a smooth proper scheme over $\F_q$ and $\varphi:X\to X$ the Frobenius morphism, we apply Hoyois' formula to $\varphi^m$. 
We can write $X^{\varphi^m}$ as a disjoint union over all of the points of degree~$X$ of degree dividing $m$: 
\begin{equation}\label{E:frob-fixed-points-disjoint-union} X^{\varphi^m} = \bigsqcup_{i \mid m} \bigsqcup_{\substack{\text{degree } i \\ \text{points of }X}} \Spec \F_{q^i}.
\end{equation}
Denote by $\alpha(i)$ the number of points of degree $i$ on $X$. Since  $d\varphi^m=0$, we obtain that
\begin{equation}\label{E:formula-Nm-traces}N_m(X):= \Tr(\varphi^m)  = \sum_{i|m} \alpha(i)\Tr_{\F_{q^i}/\F_q}\langle 1 \rangle.\end{equation}
It is classical that
\begin{equation}\label{E:classicaltrace}\Tr_{\F_{q^i}/\F_q}\langle 1 \rangle = \left\{\begin{array}{cc} i\langle 1 \rangle & \text{if $i$ odd}\\
(i-1)\langle 1 \rangle + \langle u \rangle &  \text{if $i$ even}\end{array}\right.\end{equation}
(see e.g. Lemma 58 in \cite{CubicSurface}). It remains to show that 
\begin{equation}\label{E:alphai_from_point_counts}\alpha(i) = \frac{1}{i}\sum_{d|i} \mu(d)|X(\F_{q^{i/d}})|,\end{equation}
This follows from M\"obius inversion since for every $i\geq 1$, we have 
\begin{equation*}
|X(\F_{q^i})| = \sum_{d|i} \alpha(d)d. \qedhere
\end{equation*}
\end{proof}

\begin{rmk}\label{R:compease}
 The quantity $|X(\F_{q^{i/d}})|$ in Theorem~\ref{T:enrichedzetaviaHoyois} can be computed from the eigenvalues of the Frobenius morphism on \'{e}tale cohomology groups of $X$. More precisely, choose a prime $\ell$  coprime to $q$. For $i$ such that $0 \leq i \leq 2 \dim X$, let $\{ \lambda_{i,1},\ldots,\lambda_{i,b_i} \}$ denote the eigenvalues of the Frobenius morphism on $H^i_{\textup{\'{e}t}}(X_{\overline{\F_q}},\QQ_\ell)$. Then the Grothendieck--Lefschetz trace formula (\cite[Theorem~7.1.1(ii), Section~7.5.7]{Qpoints}) tells us that
 \[ |X(\F_{q^{i/d}})| = \sum_{j=0}^{2 \dim X} (-1)^j \left( \sum_{l=1}^{b_j} \lambda_{j,l}^{i/d} \right)  .\] 
\end{rmk}

\begin{rmk}\label{R:nontrivdisc}
 From (\ref{E:classicaltrace}) and (\ref{E:alphai_from_point_counts}), we see that only the case where $q$ is odd will be interesting, and that the only contributions to $\disc N_m(X)$ come from $i \mid m$ with $i$ even. In particular, $\disc N_m(X)$ is trivial for all odd $m$ and we have the expression $N_m(X) = |X(\F_{q^m})|\langle 1 \rangle$ for all odd $m$.

For even $m$, we have the formula
\begin{equation}\label{E:discN_m_even_m} \disc N_m(X) = \sum_{\substack{i|m\\ i\ \text{even}}} \alpha (i)\end{equation}
in $\Z/2\Z$, and therefore, using (\ref{E:alphai_from_point_counts}) again,  that
\begin{equation}\label{E:discN_m_from_point_counts} \disc N_m(X) = \sum_{\substack{i|m\\ i\ \text{even}}}\frac{1}{i}\sum_{d|i} \mu(d)|X(\F_{q^{i/d}})|.\end{equation}
\end{rmk}

\begin{ex}[Enriched logarithmic zeta function of $\Spec \mathbb{F}_{q^2}$]
 Since $\Spec \mathbb{F}_{q^2}$ has exactly one closed point of degree $2$, it follows from Theorem~\ref{T:enrichedzetaviaHoyois} that
 \[ N_m(\Spec \mathbb{F}_{q^2}) = \left\{\begin{array}{cc} 0 & \text{if $m$ odd}\\
\langle 1 \rangle + \langle u \rangle &  \text{if $m$ even}\end{array}\right.,\]
 and hence
 \[ \dlog \zeta^{\Aone}_{\Spec \mathbb{F}_{q^2}, \varphi} = \sum_{\text{$m$ even}} (\langle 1 \rangle +\langle u \rangle)t^{m-1} = \langle u \rangle \dlog \frac{1}{1-\langle u \rangle t} + \dlog \frac{1}{1+t}. \]
\end{ex}

\begin{ex}[Enriched logarithmic zeta function of a twisted product of the projective line] \label{E:twistedproduct}
 Let $X \colonequals \Res_{\mathbb{F}_{q^2}/\mathbb{F}_q} \mathbb{P}^1$. The difference between the cellular structure of $\mathbb{P}^1 \times \mathbb{P}^1$ as in Example~\ref{E:productofproj} and that of $X$ is that $\Sigma_1(X) \setminus \Sigma_0(X)$ is $\mathbb{A}^1_{q} \bigsqcup \mathbb{A}^1_{q}$ for $\mathbb{P}^1 \times \mathbb{P}^1$ whereas it is $\mathbb{A}^1 \times_{\mathbb{F}_q} \Spec \mathbb{F}_{q^2}$ for $X$. 
 
 We will prove that
 \[ N_m(X) = \left\{\begin{array}{cc} 1+q_{\epsilon}^{2m} & \text{if $m$ odd}\\
1+q_{\epsilon}^{2m}+(\langle 1 \rangle + \langle u \rangle)(-\langle -1 \rangle q_\epsilon^m) &  \text{if $m$ even.}\end{array}\right. \]
 Once we have this, a direct calculation will then show that 
 \[ \dlog \zeta^{\Aone}_{X, \varphi} = \frac{d}{dt} \log \frac{1}{(1- t)(1- q_\epsilon^2t)} + \langle -u \rangle \frac{d}{dt} \log \frac{1}{1-q_{\epsilon}\langle u \rangle t} + \langle -1 \rangle \frac{d}{dt} \log \frac{1}{1+q_{\epsilon}t}. \]
 
 Since $\mathbb{P}^1 = \Spec \mathbb{F}_q \bigsqcup \mathbb{A}^1,N_m(\mathbb{P}^1) = 1-\langle -1 \rangle q_\epsilon^m, N_m(\Spec \mathbb{F}_q) = 1$ (Example~\ref{ex:projspacezeta}) and $N_m$ is a motivic measure (Proposition~\ref{N_m_motivic_measure}), it follows that $N_m(\mathbb{A}^1) = -\langle -1 \rangle q_\epsilon^m$. Once again using the fact that $N_m$ is a motivic measure, we get that
\[ N_m(\mathbb{A}^1 \times \Spec \mathbb{F}_{q^2}) = N_m(\mathbb{A}^1) \times N_m(\Spec \mathbb{F}_{q^2}) = \left\{\begin{array}{cc} 0 & \text{if $m$ odd}\\
(\langle 1 \rangle + \langle u \rangle)(-\langle -1 \rangle q_\epsilon^m) &  \text{if $m$ even}\end{array}\right.\]
We also have $\Sigma_2(X) \setminus \Sigma_1(X) = \Res_{\mathbb{F}_{q^2}/\mathbb{F}_q} \mathbb{A}^1 \cong \mathbb{A}^2$, and $\Sigma_1(X) \setminus \Sigma_0(X) = \Spec \mathbb{F}_q$.  Since $N_m(\mathbb{A}^2) = N_m(\mathbb{A}^1)^2$, and $X = \Sigma^2(X)$, putting the last few lines together, and once again using the fact that $N_m$ is a motivic measure, we obtain the formula above for $N_m(X)$.
 
The discriminant of $N_m$ is non-trivial for $m$ even. We have  \begin{equation*} \disc N_m(X) = \left\{\begin{array}{cc}1 \in \F_q^*/(\F_q^*)^2 & \text{if $m$ odd}\\
\disc \langle u \rangle \neq 1 &  \text{if $m$ even.}\end{array}\right.  \end{equation*} To see this, note that for $n$ odd $n_{\epsilon} = \frac{n-1}{2}h + \langle 1 \rangle$. Moreover, $(n_{\epsilon})^2 = n^2_{\epsilon}$, whence $\disc n_{\epsilon}^2 =1$ for $n$ odd. Furthermore, $\langle u \rangle \langle -1 \rangle n_{\epsilon} =  \frac{n-1}{2}h + \langle -u \rangle$ has discriminant $\disc  \langle -u \rangle$ for $n$ congruent to $1$ mod $4$, and $\langle -1 \rangle n_{\epsilon} =  \frac{n-1}{2}h + \langle -1 \rangle$ has discriminant $\disc  \langle -1 \rangle$ for $n$ congruent to $1$ mod $4$. Combining with the fact that $\disc$ is a homomorphism from the additive group of $\GW(\F_q)$ to $ \F_q^*/(\F_q^*)^2$ gives the claimed computation of the discriminant of $N_m$.
\end{ex}

\begin{rmk}\label{R:twisteduntwisted}
 It is easy to write down equations for the variety in Example~\ref{E:twistedproduct} explicitly. A smooth quadric $Q$ in $\mathbb{P}^3$ over $\F_q$ is isomorphic to $\mathbb{P}^1 \times \mathbb{P}^1$ over $\F_q$ if the discriminant of the corresponding bilinear form is a square in $\F_q$ (equivalently when the two rulings are defined over $\F_q$), and is isomorphic to $\Res_{\F_{q^2}/\F_q} \mathbb{P}^1$ otherwise (equivalently when the two rulings are defined over $\F_{q^2}$ but not over $\F_q$). 
 
 For example, the quadric with defining equation $11 x_0^2+x_1^2+x_2^2+x_3^2 = 0$ has non-square discriminant over $\F_{3}$ and is isomorphic to $\Res_{\F_{9}/\F_3} \mathbb{P}^1$, whereas it has square discriminant over $\F_5$ and hence is isomorphic to $\mathbb{P}^1 \times \mathbb{P}^1$ over $\F_5$.\end{rmk}

\begin{rmk}\label{R:changeinrealtop}
 The calculation in Example~\ref{E:productofproj} and Example~\ref{E:twistedproduct} illustrate the connection with the topology of the real points of a lift of $\Res_{\F_{q^2}/\F_q} \mathbb{P}^1$ to characteristic $0$ as in Proposition~\ref{pr:sigdlogA1_R}. As we remarked in the introduction, when $q$ is congruent to $3$ modulo $4$, we have $u=-1$ and the extension $\F_q \subset \F_{q^2}$ is given by $\F_{q^2} = \F_q[\sqrt{-1}]$ and the $\mathbb{R}$-schemes $\PP^1 \times \PP^1$ and $\Res_{\C/\R} \PP^1$ are lifts to characteristic zero of the varieties $\PP^1 \times \PP^1$ and $\Res_{\F_{q^2}/\F_q} \PP^1$ over $\F_q$ respectively. These varieties have natural lifts to $\Z$-schemes, and the Frobenius endomorphism also lifts, so Theorem~\ref{pr:lift_to_R} applies with $d=1$.  
 
 The $\R$-scheme $\PP^1 \times \PP^1$ has two $1$-cells, along with a $0$-cell and a $2$-cell. The $\R$-scheme $\Res_{\C/\R} \PP^1$ has only the $0$-cell and $2$-cell. Using Proposition~\ref{pr:sigdlogA1_R} and the fact that the degree of $x \mapsto x^q$ on the one-point compactification of $\R$ is $1$ for $q$ odd, we compute
 \begin{align}\label{eqn:CompareResandproduct_sign}
\sgn \dlog \zeta^{\Aone}_{\PP^1 \times \PP^1, \varphi}  = \frac{d}{dt} \log &\frac{1}{(1- t)(1- t)} - \frac{d}{dt} \log \frac{1}{(1- t)^2},  \nonumber \\
 &\textrm{ and}\\
\sgn \dlog \zeta^{\Aone}_{\Res_{\F_{q^2}/\F_q} \PP^1, \varphi}  &= \ \frac{d}{dt} \log \frac{1}{(1- t)(1- t)}  \nonumber
\end{align}
 
Note the additional $- \frac{d}{dt} \log \frac{1}{(1- t)^2}$ in $\sgn \dlog \zeta^{\Aone}_{\PP^1 \times \PP^1, \varphi}$ when compared to the $\Aone$-logarithmic zeta function of the restriction of scalars. Note also that \eqref{eqn:CompareResandproduct_sign} is consistent with Example~\ref{E:productofproj} and Example~\ref{E:twistedproduct}: the signature is a ring homomorphism. Since $q$ is odd, we have $\sgn q_\epsilon = 1$ and $\sgn q_\epsilon^2 = 1$. Since $q$ is congruent to $3$ modulo $4$, we have $u=-1$ and $\sgn \lra{u} = -1$. Combining these with the calculations in Example~\ref{E:productofproj} and Example~\ref{E:twistedproduct} we reproduce the above calculation of the signature.
\begin{align*}
\sgn \dlog \zeta^{\Aone}_{\Res_{\F_{q^2}/\F_q} \PP^1, \varphi}  &= \frac{d}{dt} \log \frac{1}{(1- t)(1- t)} + \frac{d}{dt} \log \frac{1}{1+t} + (-1) \frac{d}{dt} \log \frac{1}{1+t} \\
&= \frac{d}{dt} \log \frac{1}{(1- t)(1- t)}
\end{align*} Note that the signature and rank determine the discriminant of these $\Aone$-logarithmic zeta functions for $q$ congruent to $3$ mod $4$ by Theorem~\ref{pr:lift_to_R}, but the failure of $\disc$ to be a ring homomorphism produces a non-zero discriminant term in $\dlog \zeta^{\Aone}_{\Res_{\mathbb{F}_{q^2}/\mathbb{F}_q} \mathbb{P}^1, \varphi}$ as a result of the zero signature.

\end{rmk}

\subsection{The logarithmic zeta function of non-cellular schemes}\label{sect:elliptic_curves}

Observe that Theorem~\ref{T:enrichedzetaviaHoyois} applies to any smooth projective scheme $X$, not necessarily cellular. One may hope to directly prove that the enriched logarithmic zeta function of any smooth projective scheme is $\dlog$ rational from the formula in Theorem~\ref{T:enrichedzetaviaHoyois}, without appealing to any good underlying cohomology theory.

We illustrate Theorem~\ref{T:enrichedzetaviaHoyois} in the first interesting example of a non-cellular scheme, namely the case of an elliptic curve $E$. We will use the connection with the eigenvalues of the Frobenius endomorphism on the $\ell$-adic \'{e}tale cohomology groups as in Remark~\ref{R:compease}. The cohomology groups $H^0_{\textup{\'{e}t}}(E_{\overline{\F_q}},\QQ_\ell)$ and $H^2_{\textup{\'{e}t}}(E_{\overline{\F_q}},\QQ_\ell)$ are $1$-dimensional, and the eigenvalues of the Frobenius endomorphism are $1$ and $q$ respectively, and that $H^1_{\textup{\'{e}t}}(E_{\overline{\F_q}},\QQ_\ell)$ is $2$-dimensional with the two Frobenius eigenvalues $\lambda,\overline{\lambda}$ that satisfy $\lambda+\overline{\lambda} = a$ for some integer $a$ and $\lambda \overline{\lambda} = q$ (see \cite[Section~7.2]{Qpoints}). These imply
\[ E(\F_q^{i/d}) = 1-(\lambda^{i/d}+(\overline{\lambda})^{i/d})+q^{i/d}, \]
and in particular, that
\[ E(\F_q) = 1-a+q, \quad \quad \textup{and,} \quad \quad E(\F_{q^2}) =  1-(a^2-2q)+q^2. \]
In particular, by Theorem~\ref{T:enrichedzetaviaHoyois} the coefficient of $t^1$ in $\dlog_E^{\mathbb{A}^1}(t)$ is
\[ \frac{a-a^2+q+q^2}{2} \langle u \rangle + \frac{2-a-a^2+3q+q^2}{2} \langle 1 \rangle. \]
This coefficient has nontrivial discriminant if and only if $\frac{a-a^2+q+q^2}{2}$ is odd, or equivalently, when $q \equiv 3 \mod 4$ and $a \equiv 2,3 \mod 4$, and similarly when $q \equiv 1 \mod 4$ and $a \equiv 0,1 \mod 4$ . 

Continuing this way, for the elliptic curve with Weierstrass equation $y^2=x^3+2x+3$ over $\F_7$, which has $a = 2$, we find
\begin{align*}
\dlog_E^{\mathbb{A}^1}(t) = 6 \langle 1 \rangle t^0 &+ (59 \langle 1 \rangle + 1 \langle u \rangle) t^1 +  378 \langle 1 \rangle t^2 \\&+ 2400 \langle 1 \rangle t^3 + 16566 \langle 1 \rangle t^4 
+ (117179 \langle 1 \rangle + 1 \langle u \rangle) t^5 + \cdots
\end{align*}

\section{Motivic measures}\label{sec:motivic_measures}
For $k$ a field, we denote by $K_0(\mathrm{Var}_{k})$ the \textit{modified Grothendieck ring of varieties} over $k$, defined to be the quotient of the free abelian group on classes of algebraic varieties over $k$ by the following relations:
\begin{equation}\label{E:cut_and_paste_relation}
X - Y - U
\end{equation}
for every variety $X$ over $k$ and every closed subscheme $Y$ of $X$ with open complement $U$, and
\begin{equation}\label{E:radicial_surjective_relation}
X-Y
\end{equation}
for all varieties $X,Y$ over $k$ such that there exists a radicial surjective morphism $f:X\to Y$. Recall that a morphism is said to be \defi{radicial surjective} if it is bijective and if it induces purely inseparable extensions of residue fields.

 For $X$ a quasi-projective variety over a field $k$, Kapranov's zeta function \cite{Kapranov} is defined to be the power series with coefficients in the (modified) Grothendieck ring of varieties $K_0(\mathrm{Var}_k)$ given by
$$Z_X^{\mathrm{Kap}}(t) = \sum_{n\geq 0}[\mathrm{Sym}^n(X)] t^n,$$
where $\mathrm{Sym}^n(X)$ is the $n$-th symmetric power of $X$. When $k=\F_q$, it specializes to $\zeta_X(t)$ via the counting measure $\#_{\F_q}$. It is therefore natural to ask whether one could also recover our enriched zeta function from $Z_X^{\mathrm{Kap}}$.  We do not see an immediate way of doing this.  The natural candidate would be to apply the $\Aone$-categorical trace $\Tr(\varphi)$ where $\varphi$ denotes the Frobenius 
\[
\Tr(\varphi): K_0(\mathrm{Var}_{\F_q})\to \GW(\F_q)
\] (See Proposition~\ref{N_m_motivic_measure} to see that this is a motivic measure.) However, this gives the classical zeta function, with all coefficients in $\Z \subset \GW(\F_q)$ by \cite[Example 1.6]{hoyois2015quadratic}.  On the other hand, this motivates the question of determining whether our enriched trace and enriched zeta functions define \textit{motivic measures} in some appropriate sense, and the aim of this section is to answer this question in full.

In this section, the endomorphism $\varphi$ will be the Frobenius. 

\subsection{The $\Aone$-trace as a motivic measure}
 Recall that for every $m\geq 1$, there is a motivic measure
$$\#_{\F_{q^{m}}}: K_0(\mathrm{Var}_{\F_q})\to \Z,$$
called the \defi{counting measure}, given by sending the class $[X]$ of a variety $X$ over $\F_q$ to its point count $|X(\F_{q^m})|$.  For every $m\geq 1$, define $N_m(X) = Tr(\varphi^m)\in GW(\F_q)$ to be the $\Aone$-trace of the Frobenius endomorphism on $X$. 

\begin{pr}\label{N_m_motivic_measure} The assignment $X\mapsto N_m(X)$ induces a motivic measure
$$N_m: K_0(\mathrm{Var}_{\F_q})\to \GW(\F_q)$$
enriching $\#_{\F_q^{m}}$, in the sense that we recover $\#_{\F_q^{m}}$ by taking ranks.
\end{pr}
\begin{proof} We first show that $N_m$ is well defined. For this, note that through the group isomorphism $\GW(\F_q) \cong \Z\times \Z/2\Z$, we have $N_m(X) = (|X(\F_q)|, \disc(N_m(X)))$, with $\disc(N_m(X))$ given by formula (\ref{E:discN_m_from_point_counts}), and therefore it passes to the quotient with respect to both the cut-and-paste relations (\ref{E:cut_and_paste_relation}) and the relations (\ref{E:radicial_surjective_relation}).

To prove multiplicativity, note that the Frobenius on $X\times Y$ is given by the product $\varphi_X\times \varphi_Y$. The trace is multiplicative with respect to smash product \cite[Corollary 5.9]{PontoShulman}, giving the equality
$$N_m(X\times Y) = N_m(X)N_m(Y).$$

We may thus conclude that $N_m$ defines a motivic measure. 
\end{proof}
\subsection{The enriched zeta function as a motivic measure}
 Associating to a variety $X$ over~$\F_q$ its zeta function $\zeta_X(t)$ induces a motivic measure
$$\zeta: K_0(\mathrm{Var}_{\F_q}) \to \mathcal{R}_1$$
where $\mathcal{R}_1 = \{f\in \C(t),\ f(0) =1 \}\subset 1 + t\C[[t]]$ is equipped with the Witt ring structure (see \cite[Theorem 2.1]{Ramachandran}). The logarithmic derivative
$$\dlog: 1+ t\C[[t]] \to \C[[t]]$$
sends the Witt ring structure to the ring structure where addition is addition of power series, and multiplication is coefficient-wise multiplication. In particular, composing it with $\zeta$, we get a motivic measure
$$\dlog \zeta: K_0(\mathrm{Var}_{\F_q}) \to \C^{\N}.$$

\begin{pr} The assignment $X\mapsto \dlog \zeta^{\Aone}_{X,\varphi}(t)$ defines a motivic measure
$$\dlog \zeta^{\Aone}: K_0(\mathrm{Var}_{\F_q}) \to \GW(\F_q)[[t]]$$
lifting $\dlog \zeta$. 
\end{pr}
\begin{proof}
Follows from Proposition~\ref{N_m_motivic_measure}.
\end{proof}

\begin{rmk}\label{R:irr}
 The Kapranov zeta function $Z_X^{\mathrm{Kap}}(t)$ of any curve $X$ over $\F_q$ is a rational function. On the other hand, the Kapranov zeta function is usually not rational for varieties of higher dimensions \cite{LarsenLunts}. There is reason to believe that Morel and Sawant's conjectures \cite{MorelTalk} on $\Aone$-cellular homology for general smooth projective varieties would have consequences for rationality of the $\Aone$-logarithmic zeta function analogous to Theorem~\ref{thm:rationality_cellular_slick}.
 \end{rmk}

\bibliographystyle{amsalpha}
\providecommand{\bysame}{\leavevmode\hbox to3em{\hrulefill}\thinspace}
\providecommand{\MR}{\relax\ifhmode\unskip\space\fi MR }
\providecommand{\MRhref}[2]{%
  \href{http://www.ams.org/mathscinet-getitem?mr=#1}{#2}
}
\providecommand{\href}[2]{#2}


\end{document}